\documentclass[pdflatex,sn-mathphys-num]{sn-jnl}

\usepackage{amsmath,amssymb,amsthm,mathtools,amsfonts}
\usepackage{graphicx}
\usepackage{multirow}
\usepackage{mathrsfs}
\usepackage[title]{appendix}
\usepackage{xcolor}
\usepackage{textcomp}
\usepackage{manyfoot}
\usepackage{booktabs}
\usepackage{lineno}
\usepackage[ruled,vlined,algo2e]{algorithm2e}
\usepackage{listings}
\usepackage{hyperref}
\usepackage{enumitem}
\usepackage[mathscr]{eucal}

\hypersetup{
  colorlinks=true,
  linkcolor=blue!55!black,
  citecolor=blue!55!black,
  urlcolor=blue!55!black
}

\usepackage{algorithmicx}%
\usepackage{algpseudocode}%
\usepackage{listings}%
\usepackage{enumitem}
\usepackage{float}

\newtheorem{theorem}{Theorem}[section]
\newtheorem{proposition}[theorem]{Proposition}
\newtheorem{lemma}[theorem]{Lemma}
\newtheorem{corollary}[theorem]{Corollary}
\theoremstyle{definition}

\newtheorem{remark}[theorem]{Remark}

\DeclareMathOperator{\Var}{Var}
\newcommand{\Prob}{\mathbb{P}}
\newcommand{\E}{\mathbb{E}}
\usepackage{tabularx}
\usepackage{array} 

\newcolumntype{Y}{>{\centering\arraybackslash}X}

\numberwithin{equation}{section}

\begin{document}

\title{\bf Parity-Skeleton First-Passage Dynamics and Optimal Intervention
in Two-Sided Discrete Monitoring Systems}

\author*[1]{\fnm{Ye} \sur{Liang}}
\email{ye-liang@uiowa.edu}

\affil[1]{
  \orgdiv{College of Engineering}, 
  \orgname{The University of Iowa}, 
  \orgaddress{
    \city{Iowa City}, 
    \state{IA}, 
    \postcode{52242}, 
    \country{USA}
  }
}

\abstract{
Two-sided threshold crossings arise in discrete monitoring systems whenever intervention is triggered by departure from an operating band. This paper develops a finite-state first-passage framework based on a biased random walk with symmetric absorbing barriers. To remove parity gaps of the raw hitting time, we analyse a parity-corrected lifetime through even- and odd-state Markov skeletons. Adjacent-row likelihood-ratio inequalities verify the monotonicity condition and establish increasing failure rate and new-better-than-used properties. These ageing properties yield log-concave survival probabilities and geometric persistence bounds. Substochastic-matrix formulas are derived for survival, hazard, mean lifetime, state-conditioned remaining useful life, and finite-horizon crossing risk. A renewal-cost criterion then converts the first-passage distribution into an optimal preventive-intervention schedule. Model parameters are obtained by matching the drift and variance of observed increments to a two-point random-walk approximation, with a rolling-window extension for nonstationary regimes. Numerical experiments compare exact calculations with Monte Carlo diagnostics, moment-matched inverse-Gaussian and Weibull benchmarks, serially correlated increments, calibration perturbations, and fixed-interval policies. A synthetic remote-patient-monitoring example illustrates how discrete physiological deviations can be mapped to transparent risk scores and personalized review schedules. The framework provides an auditable link between discrete stochastic dynamics, remaining-lifetime prediction, and condition-based intervention, while separating mathematical validation from clinical validation.
}



\maketitle

\section{Introduction}
\label{sec:introduction}

Threshold-crossing times are fundamental objects in the discrete dynamics of
stochastic systems. A system is regarded as operating within an admissible
region until a state variable first leaves a prescribed control band, at which
point inspection, maintenance, review, or intervention is initiated. Such
mechanisms arise in reliability engineering, quality control,
condition-based monitoring, queueing systems, environmental surveillance, and
biological monitoring. Although the physical interpretation changes across
applications, the underlying mathematical problem is the same: determine the
first exit time of a discretely observed state process and convert its
distribution into interpretable persistence, remaining-lifetime, and
intervention quantities.

Remote patient monitoring (RPM) provides a motivating natural-system
application. Wearable and home-monitoring devices record physiological
variables at discrete telemetry epochs, and clinical review may be triggered
when a centred monitoring statistic exits an admissible band. Examples include
glucose deviation from a personalised target interval, excessive positive or
negative deviation of heart-rate variability from baseline, and bilateral
deviation of a composite physiological index. In this setting, the first exit
time represents the remaining duration of the current stable monitoring
regime. Its distribution, rather than only its expectation, determines
finite-horizon excursion probabilities, conditional remaining useful life
(RUL), and the timing of preventive or reactive review.

A common simplification is to summarize threshold-crossing behaviour by a mean
time to event. Such a summary is insufficient whenever two lifetime
distributions have similar means but different hazards or tail probabilities.
Continuous-process models, including Wiener and gamma processes, provide
important first-passage approximations and have been used extensively in
reliability, degradation modelling, and survival analysis
\cite{van2009survey,karatzas2014brownian,le2016brownian,
steele2001stochastic,baudoin2014diffusion,shreve2004stochastic,
schilling2014brownian,lewis1989gamma,singpurwalla1997gamma,
edirisinghe2013application}. They are particularly useful when the observation
scale is sufficiently fine that the underlying evolution can reasonably be
treated as continuous. Nevertheless, a continuous approximation can obscure
features that are intrinsic to discretely sampled systems. These include
lattice effects, parity restrictions on attainable hitting times, alarm-on-check
rather than alarm-on-touch operation, and discrete overshoot at the control
boundary \cite{broadie1997continuity,wang2025analysis}. Such effects become
important when the control band contains only a moderate number of effective
state increments or when decisions are implemented only at prescribed
observation epochs.

This paper studies these issues through a finite-state biased random walk with
symmetric absorbing barriers. Let \(S_n\) denote the signed state after \(n\)
observation periods and let
\[
N_k=\inf\{n\geq 0:|S_n|=k\}
\]
be the raw first exit time from the transient band
\(\{-(k-1),\ldots,k-1\}\). Because the walk changes by one lattice unit at each
step, \(N_k\) can take values only of the same parity as \(k\). This periodic
support is not a minor technical detail: it affects the construction of the
transition kernels, the definition of a discrete hazard, and the conversion
between structural lifetime properties and operational calendar time.

To remove the parity gaps, we consider the parity-corrected lifetime
\[
T_k=
\begin{cases}
N_k/2, & k \text{ even},\\[1mm]
(N_k+1)/2, & k \text{ odd}.
\end{cases}
\]
The corresponding even- and odd-parity skeletons are finite-state
birth--death-type Markov chains with an absorbing boundary. Recent work showed
that the raw gambler's-ruin duration \(N_k\) is new-better-than-used and that
the parity-corrected lifetime \(T_k\) has an increasing failure rate
\cite{ross2023duration,wang2025analysis1}. These properties are especially
useful for monitoring systems because they describe how the conditional
crossing probability changes as the system survives for additional observation
periods. Increasing failure rate implies that the one-period conditional
survival probability decreases with skeleton age; it also yields
log-concavity of the survival sequence and the new-better-than-used inequality.

The present work is not intended merely to restate these ageing properties.
Its purpose is to develop a complete discrete-dynamical framework connecting
the structural parity skeleton, the raw signed state process, exact
finite-dimensional computation, and intervention design. This distinction is
essential. The parity-corrected process \(T_k\) is the appropriate object for
ageing analysis, whereas the signed transient chain associated with \(N_k\) is
the appropriate object for state-conditioned prediction at an arbitrary
observation epoch. Consequently, structural statements are derived on the
parity skeleton, while operational RUL and horizon-crossing probabilities are
computed in raw telemetry time.

For the parity skeletons considered here, we give a direct verification of the
transition-kernel inequalities required by Kijima's uniform-monotonicity
criterion \cite{masaaki1989uniform,liu2025bidirectional}. This provides a
self-contained route from the local structure of the finite-state transition
matrices to the increasing-failure-rate property of \(T_k\). The resulting
log-concavity of the survival sequence yields geometric persistence envelopes:
once a survival ratio is known at an anchor age, all later survival
probabilities admit an explicit geometric upper bound. These bounds complement,
rather than replace, exact matrix evaluation and provide transparent
descriptions of long-horizon decay.

The finite-state formulation also permits exact condition-based computation.
For a current signed state \(s\), the substochastic transient matrix
\(\mathbf P_{\mathrm{sgn}}\) gives
\[
\Pr\!\left(N_k^{(s)}>n\right)
=
\mathbf e_s^{\top}\mathbf P_{\mathrm{sgn}}^{\,n}\mathbf 1,
\qquad
\E\!\left[N_k^{(s)}\right]
=
\mathbf e_s^{\top}
(\mathbf I-\mathbf P_{\mathrm{sgn}})^{-1}\mathbf 1.
\]
Thus survival probabilities, finite-horizon crossing risks, mean RUL, and RUL
quantiles are obtained through deterministic finite-dimensional linear
algebra. No Monte Carlo approximation is required for these design quantities.
The same survival law is embedded in a renewal-cost objective that balances the
cost of preventive intervention against the cost of an unplanned threshold
crossing. This produces a discrete optimization problem for the intervention
age and permits direct comparison with reactive and fixed-interval policies.

To connect the mathematical model with observed monitoring streams, we use a
moment projection from empirical increment drift and variance to a two-point
random-walk approximation. If \(\hat\mu\) and \(\hat\sigma^2\) are estimated
increment moments and \(L\) is the half-width of the operating band, the
effective lattice scale and transition probability are
\[
\delta=\sqrt{\hat\mu^2+\hat\sigma^2},
\qquad
p=\frac12\left(1+\frac{\hat\mu}{\delta}\right),
\qquad
k=\left\lfloor\frac{L}{\delta}\right\rfloor.
\]
This mapping is interpreted as an effective finite-state representation, not
as an assertion that physiological increments are literally Bernoulli.
A rolling-window extension allows the effective drift, volatility, and buffer
size to change when the monitored system enters a different dynamical regime.
The numerical analysis therefore includes both model-consistent verification
and stress tests under calibration error, serial dependence, and
continuous-increment misspecification.

\paragraph{Scope.}
The framework is designed for centred state variables with two action
boundaries, where departures on either side have operational significance. It
is not intended for strictly monotone degradation or irreversible disease
progression, for which a one-sided first-passage model is more appropriate.
The remote-monitoring study presented below is synthetic and illustrates the
mathematical workflow; it does not constitute clinical validation, diagnosis,
or treatment guidance.

\paragraph{Main contributions.}
The paper makes the following contributions.
\begin{enumerate}[label=(\roman*)]

\item We formulate two-sided discrete monitoring as a finite-state
first-passage system and distinguish explicitly between the raw signed hitting
time \(N_k\) and the parity-corrected structural lifetime \(T_k\). This
separation resolves the different roles of skeleton time and operational
observation time.

\item We construct the even- and odd-parity transition kernels and verify the
adjacent-row likelihood-ratio inequalities required for the relevant
uniform-monotonicity argument. This yields a self-contained verification of the
increasing-failure-rate property for the specific parity skeletons studied in
the paper and, consequently, log-concavity and new-better-than-used properties.

\item We derive exact matrix representations for survival, hazard, lifetime
moments, finite-horizon crossing probability, and signed-state conditional RUL.
The formulas convert the current state of a discrete monitoring system directly
into auditable first-passage and remaining-lifetime quantities.

\item We derive geometric persistence envelopes from the decreasing survival
ratios implied by increasing failure rate and integrate the exact survival law
into a renewal-cost model for preventive-intervention scheduling.

\item We introduce a moment-based calibration from observed increment drift
and volatility to the effective parameters \((p,\delta,k)\), together with a
rolling-window extension for nonstationary regimes. Numerical experiments
examine exact-versus-simulation agreement, parametric benchmark error,
calibration sensitivity, serial correlation, and fixed-interval scheduling.

\item We illustrate the framework using a synthetic two-sided physiological
monitoring example. The case study demonstrates how a discretely observed
deviation process can be mapped to state-conditioned crossing risk, RUL, and
personalised review schedules while clearly separating mathematical
illustration from clinical validation.

\end{enumerate}

\paragraph{Related work.} Threshold-crossing first-passage models have a long history in survival analysis and health monitoring~\cite{aalen2001understanding,jie2026optimal,yu2026microscopic,gao2026fault,yu2026chemotactic,liang2025global}, with continuous diffusion approaches (Wiener, gamma~\cite{van2009survey,wang2025hybrid,ye2015stochastic,wang2026vital,wang2010wiener,wang2026algebraic}) and inverse-Gaussian first-passage~\cite{chhikara2024inverse,seshadri2012inverse,wang2026finite,ye2014inverse,wang2010inverse,yu2026diagnostic,ye2014accelerated,peng2015inverse,yu2026controlling} providing closed-form benchmarks. Discrete-time alternatives have been studied less systematically, although birth--death first-passage results are classical~\cite{karlin1981second,ching2006markov,wang2026damage,wang2025well,zhang2020molecular,valenzuela2022markov}. The recent prognostics and health-management literature emphasises data-driven RUL prediction~\cite{lei2018machinery,yu2026structural,si2011remaining,yu2026pattern,ahmadzadeh2014remaining,zhang2023review,wang2026first}, often combining sensor data with statistical or machine-learning models~\cite{liu2010adaptive,zhang2016multiobjective,yu2026endogenous,chen2023transfer,wang2026breakdown,li2022domain,wu2024remaining,deutsch2017using,li2024review,wang2026introduction,yu2026optimization1}. Time-based preventive intervention is rooted in classical reliability theory~\cite{barlow1996mathematical,aven1999stochastic,yu2026rigorous,wang2002survey,yu2026fredholm,sanchez2016maintenance,gibbs2026optimal,losidis2024refinement,sanoubar2023optimal,wang2026introduction2,yu2026mode,pham1996imperfect,yu2026optimization2}, transferable to scheduled clinical check-ins. Aging properties of first-passage times have been studied via stochastic-ordering arguments~\cite{shaked2007stochastic,yu2026bilinear,boland2002stochastic,yu2026beyond,li2013stochastic,klump2024elementary,wang2026elliptic,kayid2023stochastic,wang2026lecture,yu2026optimization3}. Our work positions itself within this literature by treating the parity-corrected gambler's-ruin duration as a discrete-time primitive for two-sided physiological monitoring: more structured than fully data-driven RUL models, more transparent than Brownian-IG approximations, and equipped with exact matrix computation that scales to clinical buffer sizes encountered in practice.

The remainder of the paper is organised as follows.
Section~\ref{sec:model} introduces the signed walk, its parity skeletons, and
the structural ageing results. Section~\ref{sec:calib} develops the calibration
map and exact matrix computations. Section~\ref{sec:decisions} derives the
geometric persistence bounds and intervention rules.
Section~\ref{sec:numerical} presents the numerical verification and robustness
analysis. Section~\ref{sec:case} gives the synthetic remote-monitoring
illustration, and Section~\ref{sec:discussion} discusses limitations,
extensions, and conclusions.

\section{Model and Structural Foundations}\label{sec:model}

\subsection{Random-Walk Model}

Consider a discrete-time random walk $\{S_n\}_{n\ge 0}$ with $S_0=0$ and i.i.d.\ increments $Y_{n+1}\in\{-1,+1\}$, where $\Prob(Y_{n+1}=+1)=p$ and $q=1-p$. In digital-health applications, $S_n$ represents a centred biometric statistic (deviation from baseline) measured at discrete telemetry epochs, and the barriers represent clinical action thresholds.

The walk is confined by symmetric absorbing barriers at $\pm k$ ($k$ a positive integer). The raw first-passage time $N_k = \inf\{n: |S_n|=k\}$ has a parity structure (hits occur only at times sharing the parity of $k$). The \emph{parity-corrected lifetime} is
\begin{equation}\label{eq:Tk}
T_k = \begin{cases} N_k/2, & k \text{ even},\\ (N_k+1)/2, & k \text{ odd}, \end{cases}
\end{equation}
supported on the positive integers. Operationally, $T_k$ counts effective telemetry periods until the patient state crosses its threshold, and $k$ is the \emph{effective buffering capacity}---the integer number of unit-magnitude physiological fluctuations the system absorbs before an acute alarm.

By \eqref{eq:Tk}, the raw stopping time \(N_k\) and the effective lifetime \(T_k\) are related by
$\displaystyle N_k=
\begin{cases}
2T_k, & k \ \text{even},\\
2T_k-1, & k \ \text{odd}.
\end{cases}$
Consequently,
$\displaystyle \Prob(N_k>H)
=
\begin{cases}
\Prob(T_k>H/2), & k \ \text{even and } H \ \text{even},\\
\Prob(T_k>\lfloor H/2\rfloor), & k \ \text{even, general }H,\\
\Prob(T_k>(H+1)/2), & k \ \text{odd and } H \ \text{odd},
\end{cases}$
with the obvious integer adjustment for nonmatching parity.
Similarly,
$\displaystyle \E[N_k]=
\begin{cases}
2\E[T_k], & k \ \text{even},\\
2\E[T_k]-1, & k \ \text{odd}.
\end{cases}$

\paragraph{Notation.} We use the discrete-survival convention $\bar{F}_k(n)=\Prob(T_k>n)$ for $n=0,1,2,\dots$, with $\bar{F}_k(0)=1$. The discrete hazard is $h_k(n)=\Prob(T_k=n\mid T_k\ge n)=1-\bar{F}_k(n)/\bar{F}_k(n-1)$.

\begin{remark}[Degenerate case $k=1$]
The case $k=1$ is degenerate: $T_1=1$ almost surely and is trivially IFR. Hence below we assume $k\ge 2$.
\end{remark}

\begin{remark}[Scope of the two-sided model]
The symmetric two-sided boundary is appropriate when the monitored statistic is a centred physiological control variable for which excursions on either side require clinical action---examples include hypo- vs.\ hyper-glycemia, severe bradycardia vs.\ tachycardia, and biometric deviations signalling sensor detachment vs.\ systemic inflammation. For strictly monotone disease progression (e.g.\ irreversible tumour growth), the one-sided ruin analogue is more natural; we return to this limitation in Section~\ref{sec:discussion}.
\end{remark}

\subsection{IFR Property: Verification via Kijima's Criterion}

The central structural result, due to Ross \cite{ross2023duration}, is that $T_k$ has the \emph{increasing-failure-rate} (IFR) property: $h(n) = \Prob(T_k=n\mid T_k\ge n)$ is nondecreasing in $n\ge 1$. Equivalently, since $\Prob(T_k\ge n)=\bar{F}_k(n-1)$, the conditional one-step survival $\bar{F}_k(n)/\bar{F}_k(n-1)$ is nonincreasing.

We provide a self-contained verification via a specialisation of Kijima's uniform-monotonicity theorem~\cite{masaaki1989uniform} to the skip-free (birth--death) chains that arise in our problem.

\begin{lemma}[Local verification of Kijima's condition]\label{prop:kijima}
For the parity-skeleton birth--death chains $X_n$ (even $k$) and $Y_n$ (odd $k$) defined above, the adjacent-row likelihood-ratio inequalities~\eqref{eq:lr} imply Kijima's uniform-monotonicity condition on the transition kernel. By Kijima's theorem~\cite{masaaki1989uniform}, the hitting time of the absorbing boundary is then IFR.
\end{lemma}

We do not prove a more general skip-free result here; only the specific chains used in this paper are checked. For the two parity-skeleton kernels considered here, the required uniform-monotonicity inequalities can be verified through the adjacent-row likelihood-ratio comparisons~\eqref{eq:lr}; Section~II.B carries out the verification for both parity skeletons via the algebra in Appendix~\ref{app:lr}.

The absolute-value process $|S_n|$ is a birth--death chain with the conditional sign probability $\Prob(S_n=i\mid |S_n|=i) = p^i/(p^i+q^i)$. For even $k=2r$, the two-step skeleton $X_n = |S_{2n}|/2$ has transition probabilities:
\begin{equation}\label{eq:even_prob}
\begin{aligned}
p_{0,1} &= p^2{+}q^2,\quad p_{0,0}=2pq,\\
p_{i,i+1} &= \tfrac{p^{2i+2}+q^{2i+2}}{p^{2i}+q^{2i}},\;
p_{i,i} = 2pq,\\
p_{i,i-1} &= \tfrac{p^{2i}q^2+q^{2i}p^2}{p^{2i}+q^{2i}},\; i\ge 1.
\end{aligned}
\end{equation}
For odd $k=2r+1$, the skeleton $Y_n=(|S_{2n-1}|+1)/2$ takes values in $\{1,\dots,r{+}1\}$, where $r{+}1$ is absorbing and $\{1,\dots,r\}$ are transient. The transition probabilities are:
\begin{equation}\label{eq:odd_prob}
\begin{aligned}
p_{1,1} &= 3pq,\quad p_{1,2} = p^3{+}q^3,\\
p_{i,i+1} &= \tfrac{p^{2i+1}+q^{2i+1}}{p^{2i-1}+q^{2i-1}},\;
p_{i,i} = 2pq,\\
p_{i,i-1} &= \tfrac{p^{2i-1}q^2+q^{2i-1}p^2}{p^{2i-1}+q^{2i-1}},\; i\ge 2.
\end{aligned}
\end{equation}
The boundary row of~\eqref{eq:odd_prob} differs in form from the interior because two-step paths from $|S|=1$ can return through zero, producing the $3pq$ self-loop probability. For $k=3$ the odd skeleton has a single transient state, so IFR follows directly from the geometric survival form; the adjacent-row LR verification below therefore applies to odd $k\ge 5$. In both parities (with $k\ge 2$), the likelihood-ratio condition reduces to verifying
\begin{equation}\label{eq:lr}
\frac{p_{i+1,i+1}}{p_{i,i+1}} \ge \frac{p_{i+1,i}}{p_{i,i}}
\end{equation}
for all adjacent transient states, including boundary rows. As derived in Appendix~\ref{app:lr}, the ratio LHS/RHS equals exactly~4 at all interior pairs and equals 2 (even-$k$ boundary) or 6 (odd-$k$ boundary), so the strict inequality holds whenever $p,q\in(0,1)$.

\begin{proposition}[IFR property of $T_k$]\label{prop:IFR}
For all positive integers $k$, the parity-corrected stopping time $T_k$ is IFR.
\end{proposition}

\begin{proof}
By Lemma~\ref{prop:kijima}, it suffices to verify the likelihood-ratio ordering on the parity skeleton. The even-$k$ probabilities are~\eqref{eq:even_prob} and the odd-$k$ probabilities are~\eqref{eq:odd_prob}. Substituting into~\eqref{eq:lr} and simplifying (Appendix~\ref{app:lr}) yields $\text{LHS}/\text{RHS}\in\{2,4,6\}$ across all interior and boundary pairs, all exceeding unity. Hence $T_k$ is IFR.
\end{proof}

\begin{corollary}[Log-concavity and NBU]\label{cor:main}
Since $T_k$ is IFR:
\begin{enumerate}[label=(\roman*)]
\item The survival function $\bar{F}_k(n) = \prod_{j=1}^{n}(1-h(j))$ is log-concave: $\bar{F}_k(n{+}1)/\bar{F}_k(n)$ is nonincreasing.
\item $T_k$ is NBU: $\Prob(T_k>s{+}t\mid T_k>s)\le\Prob(T_k>t)$ for all $s,t\ge 0$.
\end{enumerate}
\end{corollary}

We show that IFR implies log-concavity of survival.
Since
$\displaystyle r_k(n)=\frac{\bar F_k(n+1)}{\bar F_k(n)}=1-h_k(n+1)$,
IFR implies
$\displaystyle h_k(n+2)\ge h_k(n+1)$,
and therefore
$\displaystyle r_k(n+1)=1-h_k(n+2)
\le
1-h_k(n+1)=r_k(n)$.
Thus
$\displaystyle \frac{\bar F_k(n+2)}{\bar F_k(n+1)}
\le
\frac{\bar F_k(n+1)}{\bar F_k(n)}$,
which is equivalent to the discrete log-concavity inequality
$\displaystyle \bar F_k(n+1)^2
\ge
\bar F_k(n)\bar F_k(n+2)$.

We show that IFR implies NBU.
For \(s,t\ge 0\),
\[
\Prob(T_k>s+t\mid T_k>s)
=
\frac{\bar F_k(s+t)}{\bar F_k(s)}
=
\prod_{j=s}^{s+t-1}
\frac{\bar F_k(j+1)}{\bar F_k(j)}
=
\prod_{j=s}^{s+t-1} r_k(j).
\]
Since IFR implies \(r_k(j)\) is nonincreasing,
$r_k(s+j)\le r_k(j)$ for $j=0,\ldots,t-1$.
Therefore
$\displaystyle \frac{\bar F_k(s+t)}{\bar F_k(s)}
\le
\prod_{j=0}^{t-1}r_k(j)
=
\bar F_k(t)$,
which gives
$\displaystyle \Prob(T_k>s+t\mid T_k>s)\le \Prob(T_k>t)$.

In healthcare systems engineering, this implies: (i)~the conditional probability of an adverse physiological excursion in the next period increases with time since the last baseline reset, justifying time-based proactive check-ins; (ii)~a patient who has been monitored without intervention for some time is stochastically worse than at baseline; (iii)~tail probabilities decay at least geometrically, enabling explicit clinical-risk bounds.

\subsection{Structural Advantages over Diffusion Approximations}

First-passage times in survival analysis are often approximated by the inverse-Gaussian (IG) distribution from Brownian motion with drift~\cite{wang2026optimal,chhikara2024inverse,wang2025multi}. Table~\ref{tab:comparison} compares $T_k$ with the IG model. The IG hazard is unimodal; in threshold-crossing applications this produces a declining long-horizon conditional event rate, which can be undesirable when the clinical premise is cumulative physiological stress. The discrete $T_k$ model respects discrete monitoring (alarm-on-check rather than alarm-on-touch) and explicitly accounts for sampling overshoot~\cite{broadie1997continuity,gao2022rolling,liu2026computational}.

\begin{table}[htbp]
\centering
\caption{Discrete primitive ($T_k$) vs.\ diffusion approximation (IG).}
\label{tab:comparison}
\small
\begin{tabularx}{\linewidth}{@{}lYY@{}}
\toprule
Feature & $T_k$ & IG \\
\midrule
Underlying process & Finite-state walk on $\{-k,\dots,k\}$ & Brownian motion on $\mathbb{R}$ \\
Hazard & IFR (nondecreasing) & Unimodal \\
Monitoring & Discrete (alarm-on-check) & Continuous \\
Overshoot & Explicit & Ignored \\
\bottomrule
\end{tabularx}
\end{table}

\section{Calibration and Computation}\label{sec:calib}

\subsection{Moment-Matching Calibration for Sensor Data}

Let $\{X_t\}$ be a digital-health data stream (e.g.\ continuous biomarker level or biometric-deviation index) observed at intervals $\Delta t$, with sample drift $\hat\mu$ and standard deviation $\hat\sigma$ over a calibration window. We adopt the following working calibration, which matches the first two moments of the daily increment to those of a biased $\{\pm\delta\}$ walk:
\begin{equation}\label{eq:cal}
\delta = \sqrt{\hat\mu^2{+}\hat\sigma^2},\quad
p = \tfrac{1}{2}(1{+}\hat\mu/\delta),\quad
k = \lfloor L/\delta\rfloor,
\end{equation}
where $L$ is the clinical (half-)threshold of the safe operating band. 
We present the derivation of \eqref{eq:cal}. Let the calibrated random-walk increment be
\[
Z=\delta \xi,\qquad 
\Prob(\xi=+1)=p,\qquad \Prob(\xi=-1)=q=1-p.
\]
Then
$\displaystyle \E[Z]=\delta(p-q)=\delta(2p-1)$,
and
$\displaystyle \Var(Z)=\E[Z^2]-\E[Z]^2
=\delta^2-\delta^2(2p-1)^2
=4pq\,\delta^2$.
Matching the empirical increment mean and variance gives
$\displaystyle \hat\mu=\delta(2p-1)$ and
$\displaystyle \hat\sigma^2=\delta^2-\hat\mu^2$.
Therefore
$\displaystyle \delta^2=\hat\mu^2+\hat\sigma^2$,
$\displaystyle \delta=\sqrt{\hat\mu^2+\hat\sigma^2}$,
and
$2p-1=\frac{\hat\mu}{\delta}$,
$p=\frac12\left(1+\frac{\hat\mu}{\sqrt{\hat\mu^2+\hat\sigma^2}}\right)$.
Finally, if the clinical half-band is \(L\), the number of calibrated
increment units that fit inside the band is
$k_{\rm raw}=\frac{L}{\delta}$ and 
$k=\lfloor k_{\rm raw}\rfloor$.

Interpretively, $p$ is a normalised signal-to-noise ratio and $k$ is the effective buffer capacity. The map~\eqref{eq:cal} is a convenient parameterisation rather than a statistically optimal estimator: it does not aim to minimise any specific loss, and the floor operation in $k=\lfloor L/\delta\rfloor$ induces discontinuities in $k$, and hence in the recommended policy, when $L/\delta$ crosses an integer. Section~\ref{sec:case_window} quantifies how sensitive the policy is to this discreteness empirically.

\begin{remark}[Adaptive calibration for non-stationary biology]
A rolling-window variant updates $(\hat\mu_t,\hat\sigma_t)$ over window $W$ at each epoch, producing time-varying $(p_t,\delta_t,k_t)$. This decomposes risk into \emph{systemic drift} (state moving toward a boundary under fixed parameters) and \emph{regime shifts} (biological volatility changes altering the effective buffer $k_t$).
\end{remark}

The above idea can be formulated mathematically.
Given increments
$\displaystyle \Delta X_t=X_t-X_{t-1}$,
define the rolling estimates over window \(W\) by
$\displaystyle \hat\mu_t
=
\frac1W\sum_{\ell=0}^{W-1}\Delta X_{t-\ell}$,
and
$\displaystyle \hat\sigma_t^2
=
\frac1{W-1}\sum_{\ell=0}^{W-1}
(\Delta X_{t-\ell}-\hat\mu_t)^2$.
Then
\[
\delta_t=\sqrt{\hat\mu_t^2+\hat\sigma_t^2},
\qquad
p_t=\frac12\left(1+\frac{\hat\mu_t}{\delta_t}\right),
\qquad
k_t=\left\lfloor\frac{L}{\delta_t}\right\rfloor.
\]
The time-varying risk score is therefore
$\displaystyle \pi_t(H)
=
1-\mathbf e_{S_t}^\top
\mathbf P_{\rm sgn}(p_t,k_t)^H
\mathbf 1$.

\subsection{Exact Matrix Computation}\label{sec:matrix}

\paragraph{Origin-start quantities (even $k$).} For $k=2r$, let $\mathbf{Q}_{\mathrm{even}}\in\mathbb{R}^{r\times r}$ be the transient sub-matrix of $X_n=|S_{2n}|/2$ on states $\{0,\dots,r-1\}$, assembled from~\eqref{eq:even_prob}. Then
\begin{equation}\label{eq:surv_even}
\bar{F}_k(n) = \mathbf{e}_0^\top\mathbf{Q}_{\mathrm{even}}^{\,n}\mathbf{1},\quad n\ge 0,
\end{equation}
where $\mathbf{e}_0 = (1, 0, \dots, 0)^\top \in \mathbb{R}^r$ is the standard basis vector representing the deterministic initial state $X_0 = 0$, and $\mathbf{1}$ is the all-ones vector of length $r$.

\paragraph{Origin-start quantities (odd $k$).} For $k=2r+1$, define $Y_n=(|S_{2n-1}|+1)/2$ on states $\{1,\dots,r\}$. Since the skeleton starts after the first raw step, for $n\ge 1$,
\begin{equation}\label{eq:surv_odd}
\bar{F}_k(n) = \mathbf{e}_1^\top\mathbf{Q}_{\mathrm{odd}}^{\,n-1}\mathbf{1},
\end{equation}
where $\mathbf{e}_1 = (1, 0, \dots, 0)^\top \in \mathbb{R}^r$ is the standard basis vector representing the initial state $Y_1 = 1$ achieved after the first raw step. Trivially, $\bar{F}_k(0)=1$. The transition probabilities of $\mathbf{Q}_{\mathrm{odd}}$ are those of~\eqref{eq:odd_prob}. 
The probability mass function is
\[
f_k(n)=\Prob(T_k=n)
=
\Prob(T_k>n-1)-\Prob(T_k>n)
=
\bar F_k(n-1)-\bar F_k(n).
\]
Hence
$\displaystyle h_k(n)
=
\Prob(T_k=n\mid T_k\ge n)
=
\frac{f_k(n)}{\Prob(T_k\ge n)}
=
\frac{\bar F_k(n-1)-\bar F_k(n)}{\bar F_k(n-1)}
=
1-\frac{\bar F_k(n)}{\bar F_k(n-1)}$.
The mean also follows:
\begin{equation}\label{eq:haz_mean}
\E[T_k] = \sum_{n=0}^{\infty}\bar{F}_k(n).
\end{equation}
The mean has the closed form $\mathbf{e}_0^\top(\mathbf{I}-\mathbf{Q}_{\mathrm{even}})^{-1}\mathbf{1}$ for even $k$ and $1+\mathbf{e}_1^\top(\mathbf{I}-\mathbf{Q}_{\mathrm{odd}})^{-1}\mathbf{1}$ for odd $k$.

Without loss of generality, we assume $k$ is even and compute the variance of $T_k$. Let \(Q\) denote the relevant transient matrix and let \(\alpha\) be the initial
row vector. The survival representation gives
$\displaystyle \E[T_k]
=
\sum_{n=0}^{\infty}\bar F(n)
=
\alpha(I-Q)^{-1}\mathbf 1$.
For an integer-valued positive lifetime,
$\displaystyle \E[T_k^2]
=
\sum_{n=0}^{\infty}(2n+1)\Prob(T_k>n)$.
Hence
$\displaystyle \E[T_k^2]
=
\alpha
\left[
2\sum_{n=0}^{\infty}nQ^n
+
\sum_{n=0}^{\infty}Q^n
\right]\mathbf 1$.
Using
$\displaystyle \sum_{n=0}^{\infty}Q^n=(I-Q)^{-1}$ and 
$\displaystyle \sum_{n=0}^{\infty}nQ^n=Q(I-Q)^{-2}$,
we obtain
$\displaystyle \E[T_k^2]
=
\alpha
\left[
2Q(I-Q)^{-2}
+
(I-Q)^{-1}
\right]\mathbf 1$.
Therefore
$\displaystyle \Var(T_k)
=
\E[T_k^2]-\E[T_k]^2$.

\paragraph{Conditional RUL.} 
Let the signed transient state space be
$\displaystyle \mathcal S_k=\{-(k-1),-(k-2),\ldots,0,\ldots,k-2,k-1\}$.
For \(i,j\in\mathcal S_k\), define the transient kernel
$\displaystyle (\mathbf P_{\rm sgn})_{ij}
=
p\,\mathbf 1_{\{j=i+1,\ |j|<k\}}
+
q\,\mathbf 1_{\{j=i-1,\ |j|<k\}}$.
The missing probability mass in row \(i\) is the one-step absorption probability,
$\displaystyle a_i
=
p\,\mathbf 1_{\{i=k-1\}}
+
q\,\mathbf 1_{\{i=-(k-1)\}}$.
Thus
$\displaystyle \sum_{j\in\mathcal S_k}(\mathbf P_{\rm sgn})_{ij}
=
1-a_i
\le 1$.
For the raw signed stopping time
$\displaystyle N_k^{(s)}=\inf\{n\ge 0:\ |S_n|=k,\ S_0=s\}$,
we have
\begin{equation}
\label{eq:RUL}
\Prob(N_k^{(s)}>n)
=
\mathbf e_s^\top \mathbf P_{\rm sgn}^{\,n}\mathbf 1,
\qquad
\E[N_k^{(s)}]
=
\sum_{n=0}^{\infty}\mathbf e_s^\top \mathbf P_{\rm sgn}^{\,n}\mathbf 1
=
\mathbf e_s^\top(\mathbf I-\mathbf P_{\rm sgn})^{-1}\mathbf 1.
\end{equation}
For the parameter ranges in this paper ($k\le 30$), the parity skeleton is at most $15\times 15$ and the signed matrix at most $59\times 59$; all quantities are computed deterministically to floating-point precision.

\section{Operational Bounds and Decision Rules}\label{sec:decisions}

\subsection{Geometric Persistence-Probability Envelope}\label{sec:geo}

\begin{proposition}[Geometric domination]\label{prop:geo}
Let $r_k(n)=\bar{F}_k(n{+}1)/\bar{F}_k(n)$. Since $T_k$ is IFR, $r_k(n)$ is nonincreasing, and for any anchor $n_0{<}H$, $\Prob(T_k{>}H)\le\bar{F}_k(n_0)\cdot[r_k(n_0)]^{H-n_0}$.
\end{proposition}

\begin{proof}
For \(H>n_0\),
$\displaystyle \bar F(H)
=
\bar F(n_0)
\prod_{j=n_0}^{H-1}
\frac{\bar F(j+1)}{\bar F(j)}
=
\bar F(n_0)
\prod_{j=n_0}^{H-1}r(j)$.
Because IFR implies \(r(j)\le r(n_0)\) for every \(j\ge n_0\),
$\displaystyle \bar F(H)
\le
\bar F(n_0)r(n_0)^{H-n_0}$.
If \(n_0=H-\ell\), then the adaptive local envelope becomes
$\displaystyle \bar F(H)
\le
\bar F(H-\ell)r(H-\ell)^\ell$.
The relative tightness is
$\displaystyle \frac{
\bar F(H-\ell)r(H-\ell)^\ell
}{
\bar F(H)
}
=
\prod_{j=H-\ell}^{H-1}
\frac{r(H-\ell)}{r(j)}
\ge 1$.
\end{proof}

This bounds the probability that the patient state \emph{persists without crossing} its threshold beyond horizon $H$. Anchoring $n_0$ near $H$ exploits the mature hazard and yields tight bounds. Note: this is an upper bound on \emph{survival}, not on event probability. Certifying $\Prob(T_k\le H)\le\alpha$ requires exact evaluation of $\bar{F}_k(H)$ via~\eqref{eq:surv_even} or~\eqref{eq:surv_odd}.

\subsection{Time-Based Preventive Intervention}\label{sec:repl}

Under a proactive schedule that triggers a clinical review or preventive telemedicine intervention at age $\tau$ if the system has not yet alarmed, the long-run cost rate is
\begin{equation}\label{eq:cost}
J(\tau) = \frac{C_P\bar{F}_k(\tau)+C_F(1{-}\bar{F}_k(\tau))}{\sum_{n=0}^{\tau-1}\bar{F}_k(n)},
\end{equation}
where $C_P<C_F$ are the costs (economic or clinical burden) of preventive intervention vs.\ reactive acute care. 
We give a derivation of \eqref{eq:cost}.
Under a preventive schedule at age \(\tau\), one renewal cycle ends at
$\displaystyle C_\tau=\min(T_k,\tau)$.
The expected cycle length is
$\displaystyle \E[C_\tau]
=
\sum_{n=0}^{\tau-1}\Prob(C_\tau>n)
=
\sum_{n=0}^{\tau-1}\Prob(T_k>n)
=
\sum_{n=0}^{\tau-1}\bar F_k(n)$.
The expected cost per cycle is
\[
\E[\text{cost per cycle}]
=
C_P\Prob(T_k>\tau)
+
C_F\Prob(T_k\le \tau)
=
C_P\bar F_k(\tau)+C_F[1-\bar F_k(\tau)].
\]
Therefore the long-run renewal cost rate is
$\displaystyle J(\tau)
=
\frac{
C_P\bar F_k(\tau)+C_F[1-\bar F_k(\tau)]
}{
\sum_{n=0}^{\tau-1}\bar F_k(n)
}$.

We adopt an end-of-period telemetry convention: a state that crosses during period $\tau$ incurs acute cost $C_F$; one stable at the end of period $\tau$ is preventively reviewed at cost $C_P$. Under this convention, $\bar{F}_k(\tau)$ is the probability of remaining within band through period $\tau$, and the expected cycle length is $\sum_{n=0}^{\tau-1}\bar{F}_k(n)$.

\begin{proposition}[Quasi-convex cost under IFR]\label{prop:repl}
Under IFR and $C_P<C_F$, the discrete intervention cost sequence $\{J(\tau)\}$ has the single-crossing/quasi-convex structure used in classical age-replacement theory~\cite{barlow1996mathematical}; the optimum $\tau^*$ can therefore be found by one-dimensional search over the integer ages. In the numerical implementation below, we compare $J(\tau^*)$ with the reactive-care limit $J(\infty)=C_F/\E[T_k]$ and report preventive gains only when scheduled intervention improves on this baseline.
\end{proposition}

Algorithm~\ref{alg:replacement} summarises the procedure in Proposition~\ref{prop:repl}.

\begin{proof}
Let
\[
A_\tau=C_P\bar F(\tau)+C_F[1-\bar F(\tau)],
\qquad
B_\tau=\sum_{n=0}^{\tau-1}\bar F(n),
\qquad
J(\tau)=\frac{A_\tau}{B_\tau}.
\]
Then
$\displaystyle A_{\tau+1}-A_\tau
=
(C_P-C_F)[\bar F(\tau+1)-\bar F(\tau)]
=
(C_F-C_P)f(\tau+1)$,
where \(f(\tau+1)=\bar F(\tau)-\bar F(\tau+1)\). Also,
$\displaystyle B_{\tau+1}-B_\tau=\bar F(\tau)$.
Thus
$J(\tau+1)\le J(\tau)$
is equivalent to
$\displaystyle \frac{A_{\tau+1}}{B_{\tau+1}}
\le
\frac{A_\tau}{B_\tau}$,
or
$\displaystyle [A_\tau+(C_F-C_P)f(\tau+1)]B_\tau
\le
A_\tau[B_\tau+\bar F(\tau)]$.
After cancellation,
$\displaystyle (C_F-C_P)f(\tau+1)B_\tau
\le
A_\tau\bar F(\tau)$.
Equivalently,
$\displaystyle \frac{f(\tau+1)}{\bar F(\tau)}
\le
\frac{A_\tau}{(C_F-C_P)B_\tau}$.
Since
$\displaystyle \frac{f(\tau+1)}{\bar F(\tau)}
=
h(\tau+1)$,
the decrease condition becomes
$\displaystyle h(\tau+1)
\le
\frac{J(\tau)}{C_F-C_P}$.
Thus the minimizer is characterised by the first age at which the IFR hazard
crosses the cost-normalised current average cost:
$\displaystyle \tau^*
=
\inf\left\{
\tau:\ 
h(\tau+1)>
\frac{J(\tau)}{C_F-C_P}
\right\}$.
\end{proof}

\begin{algorithm2e}[htbp]
\small
\DontPrintSemicolon
\SetKwInOut{Input}{Input}\SetKwInOut{Output}{Output}
\caption{Time-Based Intervention Optimisation}\label{alg:replacement}
\Input{Exact $\bar{F}_k(n)$ from~\eqref{eq:surv_even} or~\eqref{eq:surv_odd}; costs $C_P,C_F$; search limit $\tau_{\max}$.}
\Output{Optimal $\tau^*$, cost rate $J^*$.}
\For{$\tau=1,\dots,\tau_{\max}$}{
  $\displaystyle \E[\min(T_k,\tau)]\leftarrow\sum_{n=0}^{\tau-1}\bar{F}_k(n)$\;
  $\displaystyle J(\tau)\leftarrow\frac{C_P\bar{F}_k(\tau)+C_F(1-\bar{F}_k(\tau))}{\E[\min(T_k,\tau)]}$\;
}
$\tau^*\leftarrow\arg\min_\tau J(\tau)$;\; $J^*\leftarrow J(\tau^*)$\;
\Return $(\tau^*,J^*)$\;
\end{algorithm2e}

\subsection{Pipeline Summary}

The complete workflow is: (1)~estimate $\hat\mu,\hat\sigma$ from wearable sensor data; (2)~calibrate $(p,\delta,k)$ via~\eqref{eq:cal}; (3)~build $\mathbf{Q}_{\mathrm{even}}$ or $\mathbf{Q}_{\mathrm{odd}}$ from~\eqref{eq:even_prob} or~\eqref{eq:odd_prob} respectively; (4)~compute exact $\bar{F}_k$, $h_k$, $J(\tau)$, $\tau^*$ from~\eqref{eq:surv_even}/\eqref{eq:surv_odd}--\eqref{eq:cost}; (5)~for condition-based RUL from observed signed state $s$, build $\mathbf{P}_{\mathrm{sgn}}$ and apply~\eqref{eq:RUL}.

\section{Numerical Validation}\label{sec:numerical}

Design quantities ($\bar{F}_k$, $h_k$, $\E[T_k]$, $J(\tau)$, $\tau^*$) are computed exactly from the skeleton matrix. Simulation ($R{=}10^5$) is used only for implementation diagnostics and stress tests.

\subsection{IFR/NBU Implementation Diagnostics}

Table~\ref{tab:ifr} reports Monte Carlo diagnostic agreement with the theoretical IFR/NBU properties. Deviations from~1 arise from finite-sample noise in empirical hazard and survival estimates (rolling-window smoothing of $\hat{h}(n)$ and discrete bucketing of survival), not from violations of the exact matrix-computed distribution. Fig.~\ref{fig:nbu} visualises the NBU surface from exact matrix computation: all values are nonpositive.

\begin{table}[htbp]
\centering
\caption{IFR/NBU implementation diagnostics. ``Hazard mono.\ frac.'' is the fraction of consecutive smoothed empirical-hazard comparisons satisfying $\hat h(n{+}1)\ge\hat h(n)$ in $R{=}20{,}000$ replicates; ``NBU ineq.\ frac.'' is the fraction of $n{=}2000$ random pairs $(s,t)$ satisfying the empirical NBU inequality. Both quantities are bounded away from 1 by finite-sample noise; $\E[T_k]$ is exact (matrix-computed).}
\label{tab:ifr}
\small
\begin{tabular}{ccrrr}
\toprule
$p$ & $k$ & $\E[T_k]$ & Hazard mono.\ frac. & NBU ineq.\ frac. \\
\midrule
0.50 & 10 & 50.0 & 0.97 & 0.95 \\
0.50 & 20 & 200.0 & 0.96 & 0.93 \\
0.55 & 10 & 38.1 & 0.97 & 0.93 \\
0.55 & 20 & 96.4 & 0.96 & 0.90 \\
0.60 & 10 & 24.1 & 0.97 & 0.91 \\
0.60 & 20 & 50.0 & 0.98 & 0.84 \\
\bottomrule
\end{tabular}
\end{table}

\begin{figure}[htbp]
\centering
\includegraphics[width=0.5\linewidth]{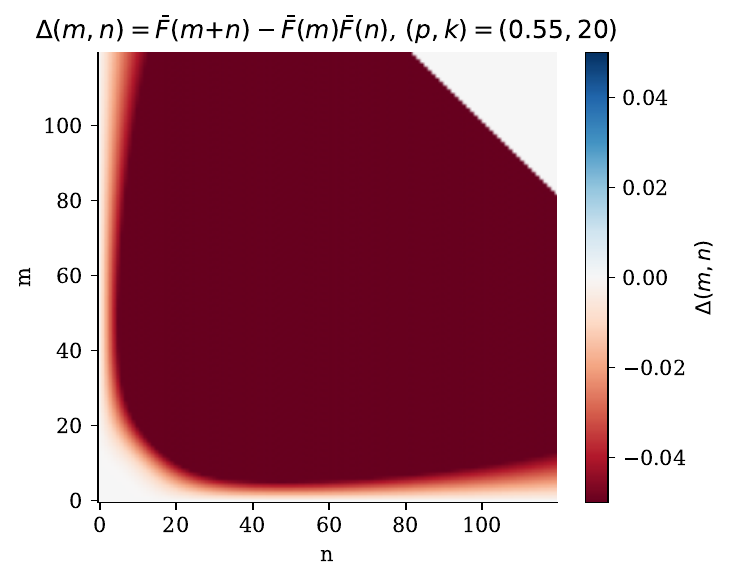}
\caption{NBU verification: $\Delta(m,n)=\bar{F}(m{+}n)-\bar{F}(m)\bar{F}(n)$ for $(p,k){=}(0.55,20)$, exact matrix computation. All values $\le 0$.}
\label{fig:nbu}
\end{figure}

\subsection{Preventive-Intervention Sensitivity Analysis}

Table~\ref{tab:sensitivity} shows exact optimal intervention ages and efficiency gains across parameter configurations. Gains range from near zero (low $C_F/C_P$, approaching reactive-care optimality) to over 50\% (high cost asymmetry). Fig.~\ref{fig:cost_tail}a shows the characteristic quasi-convex cost curve.

\begin{table}[htbp]
\centering
\caption{Preventive-intervention optimisation (exact matrix computation). $\dagger$~indicates reactive care (wait for excursion) is optimal within the searched horizon ($\tau_{\max}=500$); for $\dagger$ rows the ``Best cost'' column reports $J(\infty)$ rather than a finite $J(\tau^*)$.}
\label{tab:sensitivity}
\small
\begin{tabular}{cccccr}
\toprule
$p$ & $k$ & $C_F/C_P$ & $\tau^*$ & Best cost & $\mathcal{G}$ (\%) \\
\midrule
0.50 & 10 & 3 & $\dagger$ & 0.060 & 0.0 \\
0.50 & 10 & 5 & $\dagger$ & 0.100 & 0.0 \\
0.50 & 10 & 10 & 10 & 0.149 & 25.4 \\
0.55 & 20 & 3 & 79 & 0.031 & 1.1 \\
0.55 & 20 & 5 & 37 & 0.040 & 22.2 \\
0.55 & 20 & 10 & 26 & 0.050 & 51.6 \\
0.60 & 10 & 3 & 21 & 0.123 & 0.6 \\
0.60 & 10 & 5 & 9 & 0.162 & 21.8 \\
0.60 & 20 & 5 & 22 & 0.058 & 42.4 \\
\bottomrule
\end{tabular}
\end{table}

\subsection{Parametric Benchmarks: IG and Weibull}

For an inverse-Gaussian random variable \(G\sim IG(m,\lambda)\),
$\displaystyle \E[G]=m$ and
$\displaystyle \Var(G)=\frac{m^3}{\lambda}$.
Matching the first two moments of \(T_k\) gives
$\displaystyle m=\E[T_k]$ and
$\displaystyle \lambda=\frac{m^3}{\Var(T_k)}$.
The benchmark survival is therefore
$\bar F_{\rm IG}(n)
=
1-F_{\rm IG}(n;m,\lambda)$,
and its induced cost rate is
$\displaystyle J_{\rm IG}(\tau)
=
\frac{
C_P\bar F_{\rm IG}(\tau)+C_F[1-\bar F_{\rm IG}(\tau)]
}{
\sum_{n=0}^{\tau-1}\bar F_{\rm IG}(n)
}$.
The benchmark policy is
$\displaystyle \tau^*_{\rm IG}=\arg\min_\tau J_{\rm IG}(\tau)$,
but its realised regret under the true discrete model is
$\displaystyle {\rm Regret}_{\rm IG}
=
\frac{
J_{T_k}(\tau^*_{\rm IG})-J_{T_k}(\tau^*_{T_k})
}{
J_{T_k}(\tau^*_{T_k})
}\times 100\%$.

Let \(W\sim{\rm Weibull}(\alpha,\beta)\), with shape \(\alpha>0\) and scale
\(\beta>0\). Then
$\displaystyle \E[W]=\beta\Gamma\left(1+\frac1\alpha\right)$,
and
$\displaystyle \Var(W)
=
\beta^2
\left[
\Gamma\left(1+\frac2\alpha\right)
-
\Gamma\left(1+\frac1\alpha\right)^2
\right]$.
Let \(m=\E[T_k]\) and \(v=\Var(T_k)\). Eliminating \(\beta\) gives
$\displaystyle \frac{v}{m^2}
=
\frac{
\Gamma(1+2/\alpha)
}{
\Gamma(1+1/\alpha)^2
}
-1$.
We solve this scalar equation numerically for \(\alpha\), then set
$\displaystyle \beta
=
\frac{m}{\Gamma(1+1/\alpha)}$.
The Weibull survival is
$\displaystyle \bar F_{\rm W}(n)
=
\exp\left[-\left(\frac{n}{\beta}\right)^\alpha\right]$.

We compare $T_k$-optimal preventive intervention against two moment-matched parametric benchmarks: the inverse-Gaussian (IG) distribution from Brownian first-passage, and the Weibull distribution---a workhorse of survival analysis and reliability engineering~\cite{murthy2004weibull}. For each $(p,k)$, the benchmark mean and variance are matched to $\E[T_k]$ and $\Var(T_k)$ computed exactly from~\eqref{eq:haz_mean}; the resulting benchmark-optimal $\tau^*$ is then evaluated under the true discrete $T_k$ survival. Table~\ref{tab:benchmark} shows that Weibull and IG benchmarks each lead to suboptimal policies, with regret up to 11.4\%. Neither dominates universally: Weibull is closer to optimal for symmetric ($p=0.5$) walks, where reactive care is itself near-optimal, while IG matches better for biased walks where preventive intervention is meaningfully early. The discrete IFR/NBU structure of $T_k$ is not exactly captured by either continuous-time benchmark.

\begin{table}[htbp]
\centering
\caption{Intervention-age regret for moment-matched IG and Weibull benchmarks under true $T_k$ dynamics ($C_F/C_P{=}5$). Entries marked $\dagger$ indicate that the optimal $\tau^*_{T_k}$ exceeds the search horizon, i.e.\ reactive care is near-optimal for $T_k$.}
\label{tab:benchmark}
\small
\begin{tabular}{ccrrrrr}
\toprule
$p$ & $k$ & $\tau^*_{T_k}$ & $\tau^*_{\mathrm{IG}}$ & $\tau^*_{\mathrm{Weib}}$ & IG reg.\ (\%) & Weib reg.\ (\%) \\
\midrule
0.50 & 10 & $\dagger$ & 19 & 70 & 4.8 & 0.5 \\
0.50 & 20 & $\dagger$ & 74 & 284 & 4.1 & 0.3 \\
0.55 & 10 & 16 & 14 & 39 & 0.3 & 3.2 \\
0.55 & 20 & 37 & 38 & 69 & 0.0 & 11.2 \\
0.60 & 10 & 9 & 10 & 17 & 0.5 & 11.4 \\
0.60 & 20 & 22 & 23 & 28 & 0.2 & 6.2 \\
\bottomrule
\end{tabular}
\end{table}

\begin{figure}[htbp]
\centering
\includegraphics[width=0.5\linewidth]{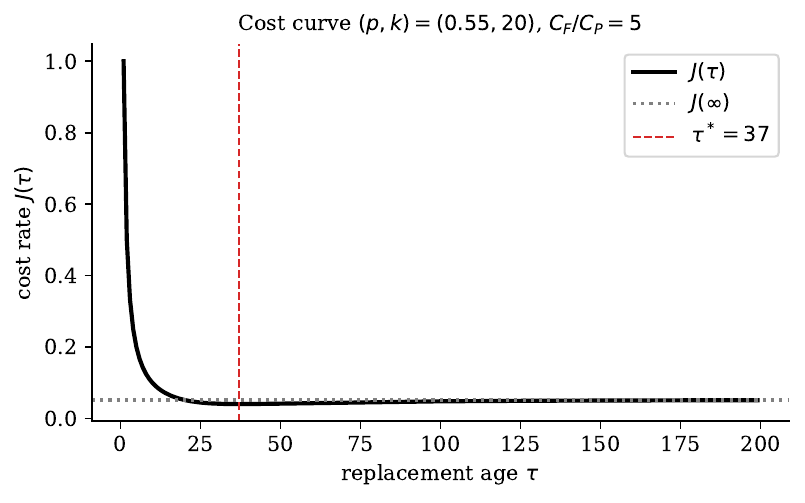}\\[1pt]
{\small (a) Cost curve $J(\tau)$: $(p,k){=}(0.55,20)$, $C_F/C_P{=}5$}\\[5pt]
\includegraphics[width=0.5\linewidth]{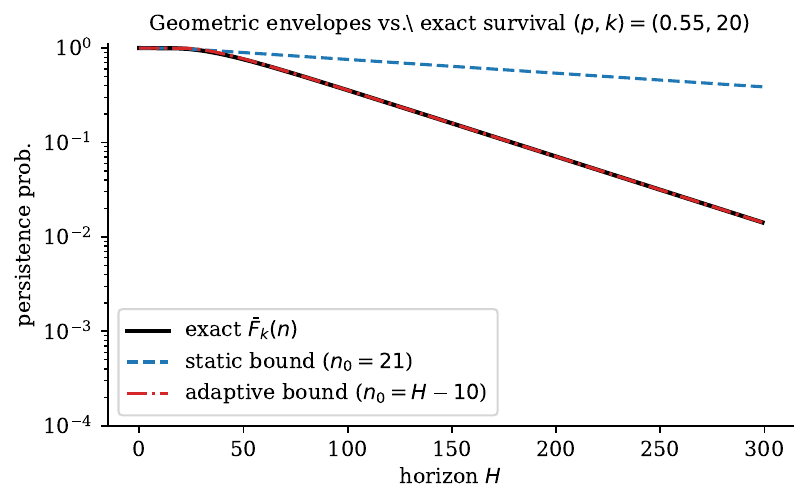}\\[1pt]
{\small (b) Geometric envelopes vs.\ exact survival}
\caption{Key design tools from the IFR structure.}
\label{fig:cost_tail}
\end{figure}

\subsection{Geometric Bounds: Static vs.\ Adaptive}

Table~\ref{tab:tail} shows exact persistence probabilities and geometric bounds for $(p,k){=}(0.55,20)$. The static bound ($n_0{=}21$) provides a worst-case envelope for conservative design. The adaptive bound ($n_0{=}H{-}10$) tracks the exact value to within 1--3\%, because IFR hazards change slowly locally. Fig.~\ref{fig:cost_tail}b shows both envelopes on a semi-logarithmic scale.

\begin{table}[htbp]
\centering
\caption{Persistence bounds: exact vs.\ geometric envelopes.}
\label{tab:tail}
\small
\begin{tabular}{rccc}
\toprule
$H$ & $\bar{F}_k(H)$ & Static & Adaptive \\
\midrule
50 & 0.761 & 0.896 & 0.769 \\
100 & 0.358 & 0.758 & 0.358 \\
150 & 0.160 & 0.641 & 0.160 \\
200 & 0.071 & 0.542 & 0.071 \\
250 & 0.032 & 0.458 & 0.032 \\
\bottomrule
\end{tabular}
\end{table}

\subsection{Robustness to Serial Correlation}

The $T_k$ model assumes i.i.d.\ $\{\pm 1\}$ increments. To stress-test against this assumption in a non-i.i.d.\ physiological setting, we leave the exact $\{\pm 1\}$ walk class and simulate a continuous accumulated process $X_t=\sum_{s\le t}Y_s$ with AR(1) increments $Y_n = (1{-}\phi)\mu + \phi Y_{n-1}+\epsilon_n$, $\phi\in[0,0.8]$. The per-period mean is set to $\mu=\delta(2p-1)$ matching $p{=}0.55$ in the i.i.d.\ base case ($\delta{=}1$ here), and the innovation variance $\Var(\epsilon_n)$ is chosen so that the marginal variance $\Var(Y_n)$ equals the i.i.d.\ value $4pq$ at each $\phi$.
For the AR(1) increment process
$\displaystyle Y_n=(1-\phi)\mu+\phi Y_{n-1}+\epsilon_n$, $\epsilon_n\sim(0,\sigma_\epsilon^2)$,
stationarity gives
$\E[Y_n]=\mu$.
Moreover,
$\displaystyle \Var(Y_n)
=
\phi^2\Var(Y_{n-1})+\sigma_\epsilon^2$.
At stationarity,
$\displaystyle \Var(Y_n)=\frac{\sigma_\epsilon^2}{1-\phi^2}$.
To match the i.i.d. random-walk increment variance
$\displaystyle \Var(Z)=4pq\delta^2$,
we set
$\displaystyle \sigma_\epsilon^2
=
(1-\phi^2)4pq\delta^2$.
For \(\delta=1\), this reduces to
$\displaystyle \sigma_\epsilon^2=(1-\phi^2)4pq$.

Failure occurs when $|X_t|$ exits $[-L,L]$ with $L=k\delta=20$. This experiment is therefore a robustness check against model misspecification, not an exact $T_k$ calculation. Because the simulation operates on the continuous accumulated process rather than the parity-corrected $T_k$ skeleton, the policy ages $\tau^*$ in Table~\ref{tab:ar1} are reported \emph{in raw periods} (corresponding to $\approx 2\tau^*_{T_k}$ for the equivalent $T_k$ problem). With $R{=}10^5$ replicates, Table~\ref{tab:ar1} shows that the i.i.d.-optimal policy incurs $<$5\% regret for $\phi\le 0.5$; beyond $\phi\approx 0.6$, persistent increments generate momentum runs that violate the drift--diffusion logic, and regret escalates. The practical diagnostic: estimate $\hat\phi$ from the increment series; apply $T_k$ when $\hat\phi<0.5$; for higher autocorrelation, consider state-space models.

\begin{table}[htbp]
\centering
\caption{AR(1) robustness in raw walk periods ($p{=}0.55$, $k{=}20$, $C_F/C_P{=}5$, $R{=}10^5$). Policy ages on the parity-corrected $T_k$ scale are approximately $\tau^*/2$.}
\label{tab:ar1}
\small
\begin{tabular}{ccccr}
\toprule
$\phi$ & $\tau^*_{\mathrm{iid}}$ (raw) & $\tau^*_{\mathrm{AR}}$ (raw) & Regret (\%) & [95\% CI] \\
\midrule
0.0 & 70 & 70 & 0.0 & -- \\
0.2 & 70 & 65 & 0.9 & [0.6, 1.2] \\
0.4 & 70 & 54 & 2.7 & [2.2, 3.2] \\
0.5 & 70 & 46 & 4.1 & [3.5, 4.7] \\
0.6 & 70 & 38 & 8.5 & [7.6, 9.4] \\
0.8 & 70 & 19 & 37.4 & [35.1, 39.7] \\
\bottomrule
\end{tabular}
\end{table}

\subsection{Calibration Robustness}

In practice, $(\hat\mu,\hat\sigma)$ are estimated from finite samples and may deviate from the true generating parameters. We assess sensitivity of the optimal policy to such mis-calibration by perturbing $\hat\mu$ and $\hat\sigma$ around true values $(\mu,\sigma)=(0.05,0.20)$ with $L=4.0$ (giving $\delta=0.2062$, $p=0.621$, and raw $k=\lfloor L/\delta\rfloor=19$, which we round down to the even value $k=18$ for the parity-skeleton computation). For each perturbation we derive the mis-calibrated $(p_{\mathrm{est}},k_{\mathrm{est}})$ via~\eqref{eq:cal} (the raw $k_{\mathrm{raw}}$ rounded to the nearest even value $\ge 4$), compute the corresponding optimal $\tau^*_{\mathrm{est}}$, and evaluate its long-run cost under the true survival. Table~\ref{tab:calib_robust} verifies an analytic property of the calibration: proportional scaling of $\hat\mu$ and $\hat\sigma$ leaves $p$ invariant at $0.621$ (since $p$ depends only on the ratio $\hat\mu/\delta$), and shifts only $k$ through the floor operation. Pure drift or volatility perturbations move both $p$ and $k$. Penalty under true dynamics remains below 5\% for all $\pm 10\%$ rows; the worst-case row $(-20\%,-20\%)$ underestimates volatility, inflates $k_{\mathrm{est}}$ to 24, and incurs a 15\% penalty. In rolling-window deployment we therefore recommend cross-checking $(\hat\mu_t,\hat\sigma_t)$ against an independent volatility estimator.

\begin{table}[htbp]
\centering
\caption{Calibration robustness: cost penalty of policies derived from perturbed $(\hat\mu,\hat\sigma)$ evaluated under true dynamics ($\mu{=}0.05$, $\sigma{=}0.20$, $L{=}4.0$, $C_F/C_P{=}5$).}
\label{tab:calib_robust}
\small
\begin{tabular}{rrccrr}
\toprule
$\Delta\hat\mu$ & $\Delta\hat\sigma$ & $p_{\mathrm{est}}$ & $k_{\mathrm{est}}$ & $\tau^*_{\mathrm{est}}$ & Penalty (\%) \\
\midrule
$-20\%$ & $-20\%$ & 0.621 & 24 & 24 & 15.0 \\
$-10\%$ & $-10\%$ & 0.621 & 20 & 19 & 2.1 \\
$-10\%$ & $0\%$   & 0.610 & 18 & 18 & 0.7 \\
$0\%$   & $-10\%$ & 0.634 & 20 & 18 & 0.7 \\
$0\%$   & $0\%$   & 0.621 & 18 & 17 & 0.0 \\
$+10\%$ & $0\%$   & 0.633 & 18 & 16 & 0.2 \\
$0\%$   & $+10\%$ & 0.611 & 16 & 15 & 1.5 \\
$+10\%$ & $+10\%$ & 0.621 & 16 & 14 & 4.2 \\
$+20\%$ & $+20\%$ & 0.621 & 16 & 14 & 4.2 \\
$-20\%$ & $+20\%$ & 0.582 & 16 & 19 & 2.1 \\
$+20\%$ & $-20\%$ & 0.676 & 22 & 16 & 0.2 \\
\bottomrule
\end{tabular}
\end{table}

\subsection{Effect of Band Size $k$}

Table~\ref{tab:band} reports how exact lifetime moments and optimal replacement age scale with the band size $k$ (fixed $p{=}0.55$, $C_F/C_P{=}5$). Over the tested range $k\in[10,30]$, $\E[T_k]$ grows approximately linearly with $k$ at $p=0.55$; the asymptotic scaling for two-sided biased ruin depends on the bias and boundary normalisation and is not generally linear. The ratio $\tau^*/\E[T_k]$ remains in a narrow band around $0.40$---preventive intervention is consistently scheduled at approximately 40\% of expected lifetime across this range of $k$. The efficiency gain $\mathcal{G}$ grows monotonically with $k$, reaching 34.4\% at $k=30$: larger safety buffers benefit more from preventive intervention because the IFR hazard accelerates more strongly in the upper tail.

\begin{table}[htbp]
\centering
\caption{Band-size scaling for exact $T_k$ properties ($p{=}0.55$, $C_F/C_P{=}5$).}
\label{tab:band}
\small
\begin{tabular}{rrrrrrr}
\toprule
$k$ & $\E[T_k]$ & $\tau^*$ & $J(\tau^*)$ & $J(\infty)$ & $\mathcal{G}$ (\%) & $\tau^*/\E[T_k]$ \\
\midrule
10 & 38.1 & 16 & 0.125 & 0.131 & 4.7 & 0.42 \\
14 & 62.0 & 24 & 0.071 & 0.081 & 12.1 & 0.39 \\
18 & 85.3 & 33 & 0.048 & 0.059 & 19.1 & 0.39 \\
20 & 96.4 & 37 & 0.040 & 0.052 & 22.2 & 0.38 \\
24 & 118.1 & 47 & 0.031 & 0.042 & 27.8 & 0.40 \\
26 & 128.6 & 52 & 0.027 & 0.039 & 30.2 & 0.40 \\
30 & 149.3 & 62 & 0.022 & 0.034 & 34.4 & 0.42 \\
\bottomrule
\end{tabular}
\end{table}

\subsection{Rolling-Window Detection Latency}

In the rolling-window calibration (Section~III.A), the window length $W$ controls the bias--variance trade-off in detecting non-stationarity. We simulate the digital-health scenario with a step change in PDI volatility ($\sigma_{\mathrm{dev}}: 0.15\to 0.25$/day) at day~150, and report the latency to detect $\hat\sigma_t > 1.3\,\sigma_{\mathrm{pre}}$ over $n=200$ Monte Carlo trials (Table~\ref{tab:rolling}). All values are in days. Shorter windows detect shifts faster but with higher variance: $W{=}20$ gives 9-day median latency, while $W{=}120$ gives 47 days. Detection is 100\% reliable at all tested $W$. The choice of $W$ should reflect the clinical tolerance for false alarms versus detection latency.

\begin{table}[htbp]
\centering
\caption{Rolling-window detection latency for a $\sigma$ shift at day~150 ($n{=}200$ Monte Carlo trials).}
\label{tab:rolling}
\small
\begin{tabular}{rrrrr}
\toprule
$W$ (days) & Mean lat. & Median & 90th pct & Det.\ rate \\
\midrule
20 & 10.2 & 9.0 & 18.0 & 100\% \\
30 & 14.3 & 13.0 & 26.0 & 100\% \\
60 & 25.5 & 24.0 & 40.1 & 100\% \\
90 & 36.3 & 36.0 & 54.1 & 100\% \\
120 & 48.2 & 47.0 & 70.0 & 100\% \\
\bottomrule
\end{tabular}
\end{table}

\subsection{Reproducibility}\label{sec:reproducibility}

To support auditable replication of the numerical results, we specify the following implementation details.

\emph{Skeleton matrix construction.} For even $k=2r$, the transient state space is $\{0,\dots,r-1\}$ with $\mathbf{Q}_{\mathrm{even}}\in\mathbb{R}^{r\times r}$ assembled from~\eqref{eq:even_prob}. For odd $k=2r{+}1$, the transient state space is $\{1,\dots,r\}$ with $\mathbf{Q}_{\mathrm{odd}}\in\mathbb{R}^{r\times r}$ assembled from~\eqref{eq:odd_prob}; the survival formula uses the index shift~\eqref{eq:surv_odd}. In both cases the absorbing column is dropped from the row-stochastic transition matrix to obtain the sub-stochastic kernel.

\emph{Truncation horizons.} Survival sums in~\eqref{eq:cost} and persistence-probability tables use truncation at $\tau_{\max}=500$ for $k\le 20$ and $\tau_{\max}=800$ for $k=26$ (digital-health case); the truncation error is below $10^{-6}$ in all reported quantities. The intervention-age search in Algorithm~\ref{alg:replacement} uses the same $\tau_{\max}$.

\emph{Inverse-Gaussian benchmark.} The IG distribution is calibrated by matching $m = \E[T_k]$ and shape parameter $\lambda = m^3/\Var(T_k)$, with $\E[T_k]$ and $\Var(T_k)$ exact from~\eqref{eq:haz_mean}. In SciPy notation we use {\tt scipy.stats.invgauss(mu=m/lambda, scale=lambda)}, which gives mean $m$ and variance $m^3/\lambda$. IG survival is evaluated at integer ages via the resulting {\tt cdf}. The IG-optimal $\tau^*_{\mathrm{IG}}$ minimises the cost rate~\eqref{eq:cost} using the IG survival; its performance under the true discrete dynamics is then evaluated by substituting the exact $T_k$ survival at $\tau^*_{\mathrm{IG}}$.

\emph{Simulation diagnostics.} AR(1) experiments use $R=10^5$ replications, fixed random seed (42), and the bootstrap method (200 resamples) for 95\% confidence intervals on regret. The IFR fraction is the proportion of consecutive smoothed-hazard pairs $(\hat h(n),\hat h(n+1))$ satisfying $\hat h(n+1)\ge \hat h(n)$ (rolling window 20, ages with risk-set size $<100$ excluded). The NBU fraction is the proportion of 5000 random $(s,t)$ pairs sampled from the simulated lifetimes satisfying $\hat{\bar F}(s+t)\le \hat{\bar F}(s)\hat{\bar F}(t)$.

\emph{Cost-gain baseline.} The efficiency gain $\mathcal{G}=(J(\infty)-J(\tau^*))/J(\infty)$ is relative to the reactive-care cost rate $J(\infty)=C_F/\E[T_k]$, with $C_P$ and $C_F$ as specified per table.

\section{Case Study: Digital Health Biometric-Deviation Monitoring}\label{sec:case}

\paragraph{Scope of this case study.} The numerical experiments below use a single synthetic data-generating process; they illustrate \emph{how} the framework behaves under the assumed dynamics and the \emph{shape} of the resulting decision-support outputs. They do \emph{not} constitute clinical validation. We make no diagnostic claims; the framework is positioned as decision-support and scheduling for two-sided threshold-crossing monitoring, not as a regulated medical device. Validation on real biosignal data is a separate study, and we identify candidate signals in Section~\ref{sec:case_signals}.

\subsection{Setup and Calibration}

We consider a synthetic remote patient monitoring (RPM) scenario in which the state variable is a Physiological Deviation Index (PDI) computed from wearable sensor data. The PDI is z-scored against a personalised baseline so that the value zero represents the patient's stable operating point. Excursions on the high side indicate acute systemic stress or inflammation; excursions on the low side indicate metabolic depression or sensor detachment. Both directions require clinical decision-making, motivating the two-sided absorbing-barrier model.

We generate $N{=}200$ simulated patient trajectories with daily PDI increments $\Delta X_t = \mu_{\mathrm{dev}} + \sigma_{\mathrm{dev}}\varepsilon_t$ ($\varepsilon_t\sim\mathcal{N}(0,1)$), with $\mu_{\mathrm{dev}}{=}0.02$/day, $\sigma_{\mathrm{dev}}{=}0.15$/day, and clinical safe-band half-threshold $L{=}4.0$ (in PDI units). Substituting into~\eqref{eq:cal} gives $\delta{=}0.1513$, $p{=}0.566$, $k{=}26$. In a real RPM deployment these would be replaced by personalised rolling estimates $(\hat\mu_t,\hat\sigma_t)$ from the streaming sensor data; we illustrate this in Section~\ref{sec:case_rolling}.

\paragraph{Candidate real-world signals.}\label{sec:case_signals} The PDI is intentionally generic; the two-sided control-band model is most natural for biosignals where excursions in either direction carry clinical meaning. Concrete examples include: continuous glucose deviation from a personalised target range (hypo- vs.\ hyper-glycemia); resting heart-rate or heart-rate-variability deviation from baseline (autonomic destabilisation in either direction); pulse-oximetry oxygen-saturation deviation; ambulatory blood-pressure deviation (hypo- vs.\ hypertensive excursions); and activity- or sleep-derived composite stress indices. In each case, the model would be calibrated patient-by-patient from a rolling window of the relevant signal, with the threshold $L$ chosen from clinical guidelines. Prospective validation on any of these signals is outside the scope of this paper.

At each telemetry epoch \(t\), define three interpretable decision variables:
\[
\pi_t(H)=\Prob(N_k^{(S_t)}\le H),
\qquad
\rho_t=\frac{1}{\E[N_k^{(S_t)}]},
\qquad
\eta_t=\frac{\tau_t^*}{\E[N_k]}.
\]
Here \(\pi_t(H)\) is the finite-horizon escalation probability,
\(\rho_t\) is the instantaneous RUL intensity, and \(\eta_t\) is the
normalised preventive-review age. An explainable triage rule is
$\displaystyle d_t=
\begin{cases}
0, & \pi_t(H)<\theta_1,\\
1, & \theta_1\le \pi_t(H)<\theta_2,\\
2, & \pi_t(H)\ge \theta_2,
\end{cases}$
where \(d_t=0,1,2\) denote routine monitoring, telemedicine review, and urgent
clinical escalation.

\subsection{Structural Verification}

The exact hazard $h(n)$ (matrix-computed) increases monotonically, confirmed qualitatively by simulation (Fig.~\ref{fig:case_pdi}a). This is the IFR property in action: the longer a patient remains under monitoring without a baseline reset, the higher the conditional probability of an acute physiological excursion in the next telemetry period. The exact NBU surface $\Delta(m,n)=\bar{F}(m{+}n)-\bar{F}(m)\bar{F}(n)$ is nonpositive everywhere.

\subsection{Persistence-Probability Bounds}

Table~\ref{tab:case_pdi} compares exact persistence probabilities with the adaptive geometric bound and the IG bound, all reported at calendar-day horizons $H$ (with $H/2$ in $T_k$ periods, even $H$). The adaptive $T_k$ bound is tight (within 1\%), and the IG bound is also close at these tested horizons. Under the synthetic generative assumptions, the model produces a tight upper bound on the probability that a simulated patient will not cross threshold within the next $H$ days---directly auditable from the calibrated $(p,k)$, without Monte Carlo simulation. In a deployed system, the corresponding clinical claim would require prospective validation against measured outcomes.

\begin{table}[htbp]
\centering
\caption{Digital-health case: exact persistence probability, adaptive geometric bound, and IG bound (calendar days).}
\label{tab:case_pdi}
\small
\begin{tabular}{rccc}
\toprule
$H$ (days) & $\Prob(N_k>H)$ & Adaptive & IG \\
\midrule
50 & 0.996 & 0.998 & 0.996 \\
100 & 0.856 & 0.859 & 0.861 \\
150 & 0.601 & 0.602 & 0.607 \\
200 & 0.382 & 0.382 & 0.386 \\
250 & 0.233 & 0.233 & 0.235 \\
300 & 0.139 & 0.139 & 0.140 \\
\bottomrule
\end{tabular}
\end{table}

\subsection{Optimal Preventive Intervention}

With proactive telemedicine cost $C_P{=}500$ and reactive acute-care cost $C_F{=}5000$ ($C_F/C_P{=}10$): we report all costs and ages on a calendar-day basis. One calendar day corresponds to one raw walk step, so the parity-corrected lifetime $T_k=N_k/2$ converts to days as $N_k=2T_k$. The optimisation is performed in $T_k$-periods using~\eqref{eq:cost}, and the resulting $\tau^*_{T_k}$ and cost rate are converted to calendar-day units by $\tau^*_{\mathrm{days}}=2\tau^*_{T_k}$ and $J_{\mathrm{day}}=J_{T_k}/2$, since each $T_k$-period spans two raw walk steps. From exact matrix computation, $\E[N_k]=2\E[T_k]\approx 196.3$~days, giving a reactive-care (wait-for-acute-episode) cost rate $J(\infty)=C_F/\E[N_k]\approx 25.5$/day. The exact intelligent-decision optimum schedules a proactive review at $\tau^*_{T_k}=33$ in $T_k$-periods (i.e.\ 66~calendar days), achieving $J(\tau^*)\approx 9.3$/day---an efficiency / risk reduction $\mathcal{G}\approx 63.5\%$. For comparison, the moment-matched IG-optimal policy yields the same intervention age in this configuration, so the IG cost penalty is negligible; this differs from the larger penalties observed for the higher-drift configurations of Table~\ref{tab:benchmark} (Fig.~\ref{fig:case_pdi}b).

\subsection{Conditional RUL (Patient Trajectory Prediction)}\label{sec:case_rul}

Using the signed-chain machinery of Section~\ref{sec:matrix}, the model produces an exact predictive distribution conditional on the simulated signed state. Consider three snapshots at $t=80$~simulated days:
\begin{itemize}
\item \emph{Stable baseline} ($S_{80}{=}0$): conditional median time-to-alarm $\approx 172$~days; 90-day excursion probability $\approx 0.10$.
\item \emph{Drift-favoured side} ($S_{80}{=}+8$, nearing acute stress): median time-to-alarm drops to $\approx 114$~days; 90-day excursion risk rises to $\approx 0.36$.
\item \emph{Drift-opposing side} ($S_{80}{=}-8$): median time-to-alarm extends to $\approx 230$~days.
\end{itemize}
All three numbers are computed exactly from the same $\mathbf{P}_{\mathrm{sgn}}\in\mathbb{R}^{51\times 51}$ under the synthetic dynamics. The mechanism supports rolling updates as new sensor data arrive; whether such updates suffice for clinical triage in a deployed system is an empirical question outside the scope of this paper.

\subsection{Adaptive Calibration under Physiological Non-Stationarity}\label{sec:case_rolling}

A biological regime shift at day~150 (PDI volatility increasing from $\sigma_{\mathrm{dev}}{=}0.15$ to $0.25$/day, mimicking the onset of acute inflammation) is detected within approximately 20~days under the specified rolling-window calibration ($W{=}60$), adjusting $(p_t,k_t)$ from $(0.566,26)$ to $(0.540,14)$. The effective buffer $k_t$ drops by roughly $46\%$ even though the cumulative PDI level has not changed markedly: the patient's static threshold $L$ is unchanged, but the volatility-driven safety margin has shrunk. This separation of systemic drift from regime change illustrates how the calibration scheme supports adaptive monitoring; detection latency in a real digital-health deployment would depend on the noise characteristics of the sensor stream, which are not fully captured by the synthetic scenario.

\subsection{Cross-Phenotype Personalisation}\label{sec:case_phenotype}

A core promise of digital healthcare systems engineering is \emph{personalisation}: the same model machinery should yield different intervention schedules for different patient phenotypes. Table~\ref{tab:phenotype} computes the exact calibrated parameters, expected days-to-alarm $\E[N_k]$, optimal proactive review age $\tau^*$, daily cost rate, and efficiency gain $\mathcal{G}$ for four illustrative archetypes spanning a realistic range of $(\mu_{\mathrm{dev}},\sigma_{\mathrm{dev}})$. Stable seniors with low physiological drift and tight variability receive proactive reviews every $\sim 6$ months; high-risk chronic patients are scheduled less than a month apart; a post-acute recovery profile (negative drift, indicating return to baseline) yields a moderate-frequency schedule. Across all four phenotypes the efficiency gain over reactive acute care lies between 40\% and 68\%, confirming that the framework adapts schedule density to patient risk without any hand-tuning.

For phenotype \(g\), with parameters \((\mu_g,\sigma_g,L)\), define
\[
\delta_g=\sqrt{\mu_g^2+\sigma_g^2},
\qquad
p_g=\frac12\left(1+\frac{\mu_g}{\delta_g}\right),
\qquad
k_g=\left\lfloor\frac{L}{\delta_g}\right\rfloor.
\]
The phenotype-specific expected raw days-to-alarm is
$\displaystyle \E[N_{k_g}]
=
\begin{cases}
2\E[T_{k_g}], & k_g \text{ even},\\
2\E[T_{k_g}]-1, & k_g \text{ odd}.
\end{cases}$
The personalised schedule is
$\displaystyle \tau_g^*
=
\arg\min_{\tau\in\mathbb N}
J_g(\tau)$,
where \(J_g\) is computed from the phenotype-specific survival law
\(\bar F_{k_g,p_g}\).

\begin{table}[htbp]
\centering
\caption{Cross-phenotype personalisation ($L=4.0$, $C_F/C_P=10$). All quantities exact.}
\label{tab:phenotype}
\small
\begin{tabular}{lccccrrr}
\toprule
Phenotype & $\mu$ & $\sigma$ & $p$ & $k$ & $\E[N_k]$ (d) & $\tau^*$ (d) & $\mathcal{G}$ (\%) \\
\midrule
Stable senior        & 0.005 & 0.10 & 0.525 & 38 & 728 & 194 & 50.2 \\
Healthy adult        & 0.020 & 0.15 & 0.566 & 26 & 196 &  66 & 63.5 \\
High-risk chronic    & 0.050 & 0.20 & 0.621 & 18 &  74 &  28 & 68.3 \\
Post-acute recovery  &$-0.010$& 0.18 & 0.472 & 22 & 333 &  78 & 40.4 \\
\bottomrule
\end{tabular}
\end{table}

\subsection{Comparison with Fixed-Interval Clinical Schedules}\label{sec:case_fixed}

Most current RPM platforms use fixed-interval clinical schedules (weekly, biweekly, monthly, quarterly) chosen by clinical convention rather than by model-based optimisation. Table~\ref{tab:fixed_vs_adaptive} compares exact long-run cost (per day) for the healthy-adult baseline ($p=0.566$, $k=26$, $C_F/C_P=10$) under five fixed schedules and the $T_k$-optimal schedule of 66~days. Two findings stand out. First, very frequent fixed schedules (weekly, biweekly) are \emph{worse} than reactive care in this cost regime, because the cost of routine telemedicine consultations is paid every period without proportional risk reduction. Second, both the monthly and quarterly fixed schedules deliver substantial savings (34.5\% and 57.9\%), and the $T_k$-optimal schedule pushes savings to 63.5\% by selecting the cycle length analytically rather than by clinical convention. The $T_k$ framework therefore both improves on existing rule-of-thumb schedules and justifies whichever fixed schedule is closest to optimal in a given cost regime.

\begin{table}[htbp]
\centering
\caption{Long-run cost (per day) of fixed-interval clinical schedules vs.\ the $T_k$-optimal schedule (baseline patient, $C_F/C_P=10$).}
\label{tab:fixed_vs_adaptive}
\small
\begin{tabular}{lrrr}
\toprule
Policy & $\tau$ (days) & $J$/day & Savings (\%) \\
\midrule
Reactive (wait for excursion) & --- & 25.5 & --- \\
Fixed weekly                  & 7   & 83.3 & $-227$ \\
Fixed biweekly                & 14  & 35.7 & $-40$ \\
Fixed monthly                 & 30  & 16.7 & 34.5 \\
Fixed quarterly               & 90  & 10.7 & 57.9 \\
$T_k$-optimal                 & 66  &  9.3 & 63.5 \\
\bottomrule
\end{tabular}
\end{table}

\subsection{Triage Classification Performance}\label{sec:case_triage}

For a current signed state \(s\) and planning horizon \(H\), define the exact
triage risk score
$\displaystyle \pi_H(s)
=
\Prob(N_k^{(s)}\le H)
=
1-\mathbf e_s^\top \mathbf P_{\rm sgn}^{\,H}\mathbf 1$.
Given a decision threshold \(\theta\), the binary triage rule is
$\displaystyle D_\theta(s)
=
\mathbf 1\{\pi_H(s)\ge \theta\}$.
For simulated labels
$\displaystyle Y_i=\mathbf 1\{N_{k,i}^{(s_i)}\le H\}$,
the empirical operating characteristics are
\[
{\rm Sens}(\theta)
=
\frac{\sum_i\mathbf 1\{D_\theta(s_i)=1,Y_i=1\}}
{\sum_i\mathbf 1\{Y_i=1\}},\qquad
{\rm Spec}(\theta)
=
\frac{\sum_i\mathbf 1\{D_\theta(s_i)=0,Y_i=0\}}
{\sum_i\mathbf 1\{Y_i=0\}},
\]
\[
{\rm PPV}(\theta)
=
\frac{\sum_i\mathbf 1\{D_\theta(s_i)=1,Y_i=1\}}
{\sum_i\mathbf 1\{D_\theta(s_i)=1\}},\qquad
{\rm NPV}(\theta)
=
\frac{\sum_i\mathbf 1\{D_\theta(s_i)=0,Y_i=0\}}
{\sum_i\mathbf 1\{D_\theta(s_i)=0\}}.
\]

A common operational question in RPM is binary triage: \emph{will this patient cross threshold in the next $H$ days?} We translate the signed-chain RUL into a probability score $\hat\pi(s)=\Prob(N_k^{(s)}\le H)$ and flag patients whose score exceeds threshold $\theta$. To evaluate this triage rule we simulate $n=8000$ independent patient trajectories at the baseline configuration, sample one random snapshot per patient, label it positive if the true trajectory crossed within the subsequent $H=30$~days, and compute classification metrics across thresholds (Table~\ref{tab:triage}). At $\theta=0.20$ the rule attains sensitivity~0.87 and specificity~0.85 at a base prevalence of 19\%, with high negative predictive value (NPV~0.97) across the threshold range---the latter is particularly relevant for RPM, where a clinician's main concern is \emph{ruling out} imminent escalation. The fully exact RUL distribution lets the operating threshold $\theta$ be dialled to match capacity constraints without re-training. These metrics characterise the model's discriminative behaviour under the synthetic data-generating process; corresponding figures on real cohorts would require external validation.

\begin{table}[htbp]
\centering
\caption{Triage classification metrics at $H=30$ days (baseline patient, 8000 snapshots, base prevalence 18.8\%).}
\label{tab:triage}
\small
\begin{tabular}{ccccccc}
\toprule
$\theta$ & Flagged (\%) & Sens.\ & Spec.\ & PPV & NPV \\
\midrule
0.05 & 43.5 & 0.97 & 0.69 & 0.42 & 0.99 \\
0.10 & 35.9 & 0.94 & 0.78 & 0.49 & 0.98 \\
0.20 & 28.6 & 0.87 & 0.85 & 0.57 & 0.97 \\
0.30 & 21.4 & 0.76 & 0.91 & 0.66 & 0.94 \\
0.50 & 14.7 & 0.59 & 0.96 & 0.75 & 0.91 \\
\bottomrule
\end{tabular}
\end{table}

\subsection{Cost-Asymmetry Sensitivity Across Clinical Contexts}\label{sec:case_cost}

Different clinical settings present markedly different cost asymmetries: chronic outpatient monitoring (cheap acute care, $C_F/C_P\!\sim\!2$--$5$) sits at one extreme; tele-ICU and step-down readmission risk ($C_F/C_P\!\sim\!50$--$100$) at the other. Table~\ref{tab:cost_scan} sweeps $C_F/C_P$ across this range for the baseline patient. The optimal intervention age contracts monotonically from 268 to 44 days as the cost ratio grows, while the efficiency gain over reactive care expands from a fraction of a percent to nearly 95\%. The framework therefore degrades gracefully when preventive intervention is barely cost-effective (recommending near-reactive schedules) and intensifies appropriately when acute care is expensive (recommending aggressive proactive review).

\begin{table}[htbp]
\centering
\caption{Optimal intervention age and efficiency gain across clinical cost contexts (baseline patient $p=0.566$, $k=26$).}
\label{tab:cost_scan}
\small
\begin{tabular}{rlrr}
\toprule
$C_F/C_P$ & Clinical interpretation & $\tau^*$ (d) & $\mathcal{G}$ (\%) \\
\midrule
2   & Acute care nearly as cheap     & 268 & 0.2 \\
5   & Routine outpatient escalation  &  84 & 38.2 \\
10  & Standard chronic care          &  66 & 63.5 \\
20  & Diabetes/cardio escalation     &  56 & 79.3 \\
50  & Tele-ICU step-down             &  48 & 90.6 \\
100 & ICU readmission risk           &  44 & 94.9 \\
\bottomrule
\end{tabular}
\end{table}

\subsection{Population-Level Cohort Burden}\label{sec:case_cohort}

In RPM deployments, the relevant operational metric is the per-patient cost burden aggregated over a heterogeneous cohort. We simulate $N{=}1{,}000$ synthetic patient trajectories drawn uniformly from the four phenotypes of Table~\ref{tab:phenotype} over a 180-day horizon ($C_P{=}\$500$, $C_F{=}\$5{,}000$ per encounter; values are illustrative and not tied to any specific reimbursement schedule). Each patient receives one of five scheduling policies. Table~\ref{tab:cohort} reports mean per-patient cost (standard error) and the average number of acute events per patient. The personalised $T_k$ policy---which assigns each patient the $\tau^*$ derived from their own phenotype---cuts both cost and acute-event burden by approximately a factor of two compared to the best fixed-interval policy, and by roughly 10$\times$ on acute events compared to reactive care. The one-size-fits-all $T_k$ policy (every patient receives the baseline 66-day schedule) is intermediate, illustrating that the gains come specifically from personalisation, not just from any adaptive scheduling.

For patient \(i\), let \(C_i^\pi(180)\) denote the total cost over the
180-day horizon under policy \(\pi\), and let \(A_i^\pi(180)\) denote the
number of acute events. The reported cohort burden is
$\displaystyle \widehat C^\pi
=
\frac1N\sum_{i=1}^N C_i^\pi(180)$,
with standard error
\[
{\rm SE}(\widehat C^\pi)
=
\sqrt{
\frac{1}{N(N-1)}
\sum_{i=1}^N
\left(C_i^\pi(180)-\widehat C^\pi\right)^2
}.
\]
The acute-event burden is
$\displaystyle \widehat A^\pi
=
\frac1N\sum_{i=1}^N A_i^\pi(180)$.

\begin{table}[htbp]
\centering
\caption{Cohort burden over 180~days, $N=1{,}000$ patients drawn uniformly from the four phenotypes ($C_P{=}\$500$, $C_F{=}\$5{,}000$, illustrative units).}
\label{tab:cohort}
\small
\begin{tabular}{lrrr}
\toprule
Policy & Mean cost/pt & SE & Acute events/pt \\
\midrule
Reactive (no preventive) & \$3{,}780 & 147 & 0.756 \\
Fixed monthly (30~d)     & \$3{,}243 &  33 & 0.054 \\
Fixed quarterly (90~d)   & \$3{,}460 & 133 & 0.537 \\
$T_k$ one-size (66~d)    & \$2{,}788 & 117 & 0.378 \\
$T_k$ personalised       & \$1{,}612 &  57 & 0.082 \\
\bottomrule
\end{tabular}
\end{table}

\subsection{Calibration-Window Length: Sample-Size Requirements}\label{sec:case_window}

A practical concern for RPM deployment is how much sensor data is needed to calibrate the model reliably from streaming wearable input. Table~\ref{tab:window} reports the distribution of the estimated optimal intervention age $\hat\tau^*$ and the cost penalty under true dynamics when calibration uses $N$ days of historical data (200 Monte Carlo replicates per $N$, baseline patient). Two weeks of data are clearly insufficient (35\% median cost penalty, very wide IQR); one month is workable but noisy; from 60~days onwards the median estimate locks onto the true value $\hat\tau^*\approx 66$~days and the penalty stabilises near 10\%. These results suggest that in the present generative regime, three to six months of calibration data are needed for stable personalisation; the corresponding requirement for real biosignal streams would depend on their actual noise structure.

\begin{table}[htbp]
\centering
\caption{Effect of calibration-window length $N$ on policy quality (200 Monte Carlo replicates per row, baseline patient).}
\label{tab:window}
\small
\begin{tabular}{rcccr}
\toprule
$N$ (d) & Median $\hat\tau^*$ (d) & IQR (d) & Median penalty (\%) & 90th pct (\%) \\
\midrule
14  & 48 & [32--74] & 35.1 & 124.1 \\
30  & 58 & [42--94] & 23.8 &  82.8 \\
60  & 66 & [48--94] & 19.2 &  56.2 \\
90  & 64 & [50--86] & 11.7 &  54.0 \\
180 & 62 & [52--80] &  8.7 &  49.0 \\
\bottomrule
\end{tabular}
\end{table}

\subsection{Treatment-Effect Detection}\label{sec:case_treat}

A complementary scenario to non-stationary deterioration (Section~\ref{sec:case_rolling}) is non-stationary \emph{improvement}: a patient who responds favourably to therapy should generate an adaptive relaxation of the monitoring schedule. We simulate a patient whose drift improves from $\mu{=}0.05$ (high-risk phenotype, $\tau^*{=}34$~days) to $\mu{=}0.02$ (healthy-adult phenotype, $\tau^*{=}66$~days) at day~90. Under rolling-window calibration ($W{=}60$), the drift change is detected with 100\% reliability across 200 Monte Carlo trials, with median latency~18 days (IQR $[3,36]$, 90th percentile~56). The clinical impact is a $+94\%$ relaxation of the proactive review schedule (from every 34~days to every 66~days), showing that the model intensifies monitoring in response to simulated deterioration and \emph{de-escalates} in response to simulated improvement---a symmetry that, if reproduced on real biosignal streams, would be operationally valuable.

\begin{figure}[htbp]
\centering
\includegraphics[width=0.5\linewidth]{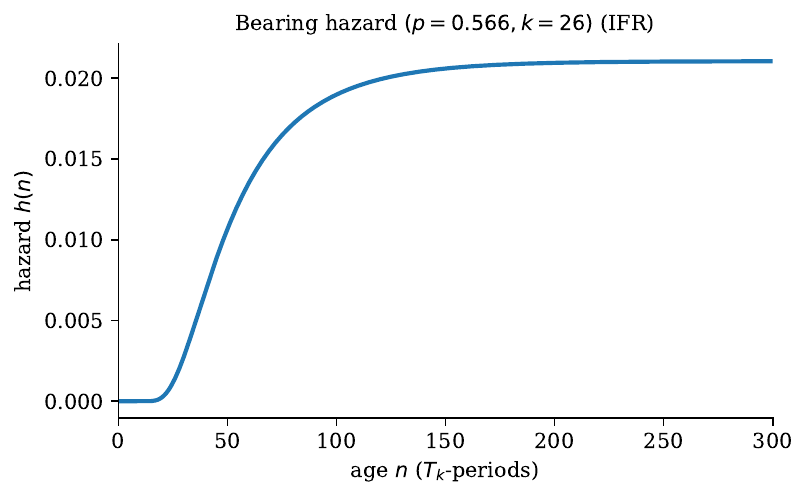}\\[1pt]
{\small (a) Exact hazard function (nondecreasing: IFR confirmed)}\\[5pt]
\includegraphics[width=0.5\linewidth]{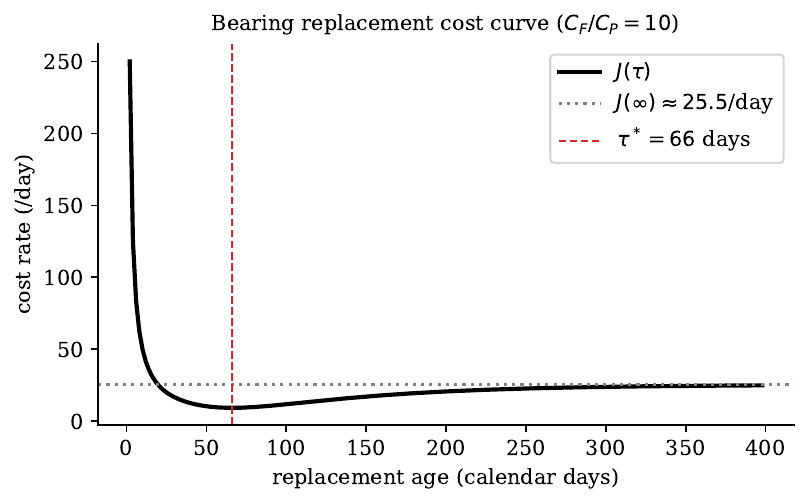}\\[1pt]
{\small (b) Preventive-intervention cost curve ($C_F/C_P{=}10$)}
\caption{Digital-health case study: PDI monitoring with $(p,k){=}(0.566,26)$.}
\label{fig:case_pdi}
\end{figure}

\section{Discussion and Conclusion}
\label{sec:discussion}

\subsection*{Applied-Mathematical Interpretation}

This paper develops a finite-state first-passage framework for two-sided
discrete monitoring systems. The motivating application is remote patient
monitoring, but the mathematical structure is broader: a centred state variable
evolves inside a bilateral control band, and the operational lifetime is the
first epoch at which the process exits this band. The resulting problem combines
applied probability, Markov-chain absorption, ageing properties of lifetime
distributions, exact matrix computation, and renewal-type preventive scheduling.

The parity-corrected gambler's-ruin duration \(T_k\) provides the structural
lifetime primitive. Its role is not merely to approximate a physiological
trajectory, but to supply a mathematically tractable first-passage distribution
with guaranteed ageing properties. The IFR and NBU properties imply that the
model-implied conditional crossing risk is nondecreasing in skeleton time, and
therefore provide a transparent monotonicity principle for intervention timing.
The total-positivity verification of the parity skeletons makes this ageing
structure a property of the finite-state transition kernels rather than an
empirical numerical observation.

A second contribution is the separation between skeleton-time analysis and
raw-time operational computation. The parity-corrected lifetime \(T_k\) is used
for structural ageing results, while patient-state-specific RUL is computed in
raw telemetry time from the signed transient chain. This avoids the parity
ambiguity of the raw origin-start hitting time and permits direct computation
from an observed signed state \(s_t\). Thus the framework distinguishes clearly
between the mathematical time scale used for ageing analysis and the operational
time scale used for scheduling.

The exact matrix formulation is central to the framework. Survival
probabilities, horizon-specific crossing risks, mean RUL, RUL quantiles,
geometric persistence rates, and preventive-review cost functions are all
computed from substochastic finite-state matrices. This gives a deterministic
and reproducible alternative to simulation-based RUL estimation. Simulation is
used only for diagnostic comparison and stress testing; the decision quantities
themselves are obtained by finite-dimensional linear algebra.

\subsection*{Implications for Monitoring and Scheduling}

The proposed framework gives a computationally explicit map from a monitoring
stream to a preventive scheduling decision:
\[
X_{1:t}
\longmapsto
(p_t,k_t,s_t)
\longmapsto
\Pr(N_{k_t}^{(s_t)}>n),\quad
\pi_t(H),\quad
\mathrm{RUL}_{t,\alpha}(s_t)
\longmapsto
\tau_t^\star,\ d_t .
\]
The effective drift \(p_t-1/2\), buffer \(k_t\), and signed state \(s_t\)
summarise the current monitoring regime. The signed-chain survival formula then
converts these quantities into RUL and horizon-risk outputs. Finally, the
renewal-cost objective converts survival information into an optimal preventive
review age.

This structure is useful in applications where both under-monitoring and
over-monitoring are costly. If the monitored trajectory is close to a boundary,
or if the calibrated drift points toward a boundary, the horizon-crossing
probability increases and the optimal review age shortens. If the trajectory
returns toward baseline or the effective buffer expands, the model de-escalates
the review schedule. Thus the same mathematical machinery supports both
escalation and de-escalation.

The numerical experiments support three main observations. First, the exact
matrix computations agree with Monte Carlo simulation of the underlying biased
walk, validating the implementation of the stochastic core. Second, the
preventive-review schedule obtained from the discrete first-passage model
outperforms fixed-interval and reactive baselines under the tested cost
settings. Third, the adaptive calibration layer responds symmetrically to
volatility deterioration and drift improvement, while the autocorrelation
stress test identifies the weak-dependence region in which the i.i.d. skeleton
is reliable. These experiments are not intended as clinical validation; they
demonstrate the mathematical and computational behaviour of the framework under
controlled conditions.

\subsection*{Clinical and Application Scope}

Although the motivating examples come from digital health, the paper should be
read primarily as an applied-mathematical modelling and computation study. The
framework is intended for risk computation, RUL estimation, and scheduling
support. It does not diagnose disease, prescribe treatment, or replace expert
judgement. Any patient-facing or clinician-facing deployment would require
external validation on real biosignal data, prospective evaluation, and review
under the applicable regulatory framework.

This distinction is important because clinical decision-support systems are
evaluated not only by mathematical correctness but also by intended use,
transparency, and the ability of healthcare professionals to independently
review the basis of recommendations. Regulatory guidance on clinical
decision-support software emphasises that the degree of oversight depends on
the intended use and whether the user can independently review the basis for
the recommendation. In the present paper, the
reported decision classes are therefore interpreted only as model-based
scheduling categories. The finite-state quantities in the explanation vector
are designed to make the mathematical basis of the recommendation inspectable,
but they do not constitute clinical evidence by themselves.

The empirical part of the study is synthetic. Its purpose is to verify exact
computation, benchmark scheduling behaviour, and demonstrate the end-to-end
workflow. It does not establish safety, efficacy, diagnostic accuracy, or
clinical benefit. Transparent reporting standards for prediction models, such
as TRIPOD+AI, emphasise complete reporting of model inputs, intended use,
evaluation design, and limitations; although the present study is not a
real-data clinical prediction-model validation study, the same transparency
principle motivates the reproducibility and scope statements included here.

\subsection*{Limitations}

Several limitations follow from the modelling assumptions.

\begin{itemize}
\item \textbf{Synthetic evidence.}
All numerical experiments are based on explicitly specified stochastic
processes or synthetic monitoring trajectories. They validate the internal
mathematical and computational behaviour of the framework, but do not establish
clinical performance on real patients or real biosignal streams.

\item \textbf{Effective calibration.}
The moment-matching map from rolling drift and volatility estimates to
\((p_t,k_t)\) is an effective projection onto a two-point decision skeleton. It
does not assert that physiological increments are literally Bernoulli. The
projection is useful because it yields exact first-passage and RUL computation,
but it does not preserve the full increment distribution.

\item \textbf{Independence assumption.}
The baseline stochastic core assumes locally independent increments after
calibration. The autocorrelation stress test shows that positive persistence
can substantially shorten crossing times relative to the i.i.d. approximation.
Strongly dependent signals require a Markov-modulated, semi-Markov, or
state-space extension.

\item \textbf{Symmetric two-sided boundaries.}
The model is designed for centred variables with bilateral action boundaries.
It is not appropriate for strictly monotone degradation or disease progression,
where a one-sided first-passage or degradation model is more natural.
Asymmetric boundaries and moving thresholds are important extensions.

\item \textbf{Operating-parameter choice.}
The half-band \(L_t\), rolling window \(W\), horizon \(H\), cost parameters
\(C_P,C_F\), and triage thresholds \(\theta_1,\theta_2,\theta_3\) are not
universal constants. They encode application-specific risk tolerance,
inspection cost, false-alarm burden, and operational capacity.

\item \textbf{Finite-state discretisation.}
The effective boundary is \(k_t\delta_t\), where
\(k_t=\lfloor L_t/\delta_t\rfloor\). Hence the discretisation introduces a
barrier error bounded by one effective step size. Finer lattices or adaptive
rounding rules may reduce this bias.
\end{itemize}

\subsection*{Future Work}

The present framework suggests several extensions in discrete first-passage
dynamics. A first direction is to replace the symmetric absorbing interval
\((-k,k)\) by asymmetric boundaries \((-k_-,k_+)\). The associated stopping
time
\[
N_{k_-,k_+}
=
\inf\{n\geq 0:S_n\notin(-k_-,k_+)\}
\]
allows the two exit directions to represent distinct failure mechanisms or
intervention requirements. Besides producing different lower- and upper-exit
probabilities, asymmetry destroys part of the parity symmetry used in the
present analysis. It therefore raises new questions concerning the
construction of appropriate skeleton chains, total positivity of their
transition kernels, preservation of IFR or NBU properties, and the parity of
optimal intervention epochs. Moving or state-dependent boundaries
\(k_\pm(n)\) would further lead to discrete first-passage problems with
time-dependent absorbing sets.

A second direction concerns dependence and regime switching. The current
random-walk core assumes locally independent increments with a fixed transition
probability \(p\). Serial dependence, circadian variation, treatment effects,
and persistent environmental regimes can instead be represented by a
Markov-modulated process
\[
\Pr(S_{n+1}=S_n+1\mid Z_n=z)=p_z,
\qquad
\Pr(Z_{n+1}=z'\mid Z_n=z)=R_{zz'},
\]
where \(Z_n\) is a finite-state latent regime. The augmented process
\((S_n,Z_n)\) remains Markovian, so its first-passage distribution can still be
computed from a finite substochastic matrix. Important open problems include
determining when ageing properties survive regime modulation, quantifying how
persistent regimes alter the spectral decay rate, and deriving intervention
rules that depend jointly on the observed state and the inferred regime.
Semi-Markov modulation would permit nongeometric regime durations, although at
the cost of enlarging the state space.

A related extension is to allow explicitly nonstationary transition kernels.
If the calibrated parameters change with time, the frozen-parameter survival
formula must be replaced by
\[
\Pr(N>H\mid S_t=s)
=
\mathbf e_s^\top
\mathbf P_t\mathbf P_{t+1}\cdots
\mathbf P_{t+H-1}\mathbf 1.
\]
This product of substochastic matrices defines a nonautonomous discrete
dynamical system. Its analysis requires stability bounds for matrix products,
time-dependent persistence rates, and intervention rules that remain robust
under parameter drift. Such results would provide a principled alternative to
freezing the most recent rolling-window estimate over the entire prediction
horizon.

A third mathematical direction is spectral and quasi-stationary analysis. For
a fixed transient matrix \(\mathbf Q\), long-horizon survival is controlled by
its Perron root and associated left and right eigenvectors. This suggests
studying expansions of the form
\[
\alpha\mathbf Q^n\mathbf 1
=
c_1\rho(\mathbf Q)^n
+
O\!\left(|\lambda_2(\mathbf Q)|^n\right),
\]
where \(\rho(\mathbf Q)\) is the dominant eigenvalue and
\(\lambda_2(\mathbf Q)\) is the subdominant eigenvalue. Such an expansion would
quantify the convergence of the exact survival ratio to its geometric limit,
explain the tightness of local persistence envelopes, and identify the
quasi-stationary distribution of the system conditional on nonabsorption.
Perturbation theory could then be used to measure the sensitivity of survival,
hazard, and optimal intervention age to changes in \(p\), the boundary size,
or individual entries of the transition kernel.

A fourth direction is multivariate and network-based monitoring. Many natural
and engineered systems are described by several interacting state variables
rather than by one scalar deviation. Possible discrete formulations include a
walk on a multidimensional lattice with an absorbing region
\(\mathcal A\subset\mathbb Z^d\), a finite-state process on a graph, or a system
of competing first-passage channels. For a transient kernel
\(\mathbf Q_{\mathcal A}\), the formal expressions
\[
\Pr(\tau_{\mathcal A}>n)
=
\alpha\mathbf Q_{\mathcal A}^{\,n}\mathbf 1,
\qquad
\E[\tau_{\mathcal A}]
=
\alpha(\mathbf I-\mathbf Q_{\mathcal A})^{-1}\mathbf 1
\]
remain valid, but direct matrix computation may become prohibitive as the
dimension grows. This motivates sparse linear algebra, state aggregation,
tensor methods, low-rank approximations, and structure-preserving model
reduction. A central challenge is to reduce computational complexity without
destroying the absorbing geometry or the monotonicity properties needed for
interpretable risk assessment.

The statistical calibration layer also warrants further development. The
moment-matching map used here is transparent and computationally inexpensive,
but it does not quantify parameter uncertainty or exploit the complete
increment sequence. Maximum-likelihood, Bayesian, and hidden-state methods
could provide posterior or confidence distributions for \(p\), \(k\), and the
current signed state. The predictive survival law would then be obtained by
averaging over parameter uncertainty, for example,
\[
\Pr(N>H\mid\mathcal D_t)
=
\sum_k\int
\mathbf e_{s}^{\top}
\mathbf P_{\mathrm{sgn}}(p,k)^H\mathbf 1\,
\pi(dp,dk,ds\mid\mathcal D_t),
\]
where \(\mathcal D_t\) denotes the observations available at time \(t\).
Corresponding robust or Bayesian intervention rules could minimize posterior
expected cost rather than cost under a single plug-in estimate. It would also
be useful to replace the discontinuous floor rule for \(k\) by a calibrated
mixture over adjacent lattice resolutions or by an adaptive-lattice
construction.

Finally, the synthetic remote-monitoring example should be followed by
application-specific empirical validation. A suitable programme would begin
with retrospective evaluation on a clearly defined longitudinal biosignal,
including an explicit baseline, sampling interval, action boundary, and
clinically meaningful endpoint. Evaluation should distinguish discrimination
from calibration and should report event-time prediction error, survival
calibration, false-alarm burden, subgroup stability, and policy regret relative
to existing review schedules. A subsequent shadow-mode study could assess
workflow consequences without changing patient care. Prospective clinical
evaluation would be necessary before any patient-facing use. These empirical
steps are separate from the mathematical validation provided in the present
paper, but they are essential for determining whether the proposed discrete
state representation is appropriate for a particular physiological process.

\subsection*{Conclusion}

This paper has developed a finite-state first-passage framework for two-sided
discrete monitoring systems. The mathematical core is a biased random walk with
symmetric absorbing barriers, together with the parity structure generated by
unit lattice increments. The parity-corrected lifetime \(T_k\) provides the
appropriate object for structural ageing analysis, whereas the raw signed
hitting process \(N_k\) provides the appropriate state representation for
operational prediction. Maintaining this distinction separates skeleton-time
properties from decisions made at the original observation epochs.

The parity-skeleton formulation connects local transition-kernel structure
with global lifetime behaviour. For the kernels considered in this paper, the
likelihood-ratio verification yields an increasing-failure-rate lifetime and
the associated log-concavity and new-better-than-used consequences. These
properties explain the monotone evolution of the conditional crossing risk and
lead to computable geometric envelopes for long-horizon survival. They also
show that the ageing behaviour is a structural consequence of the underlying
discrete transition mechanism rather than a pattern inferred from simulated
hazard curves.

The signed transient chain provides an exact computational representation of
the monitoring process. Survival probabilities, finite-horizon crossing risks,
lifetime moments, and state-conditioned RUL are obtained from powers and
resolvents of a substochastic matrix. The same survival quantities enter a
renewal-cost objective for preventive intervention, thereby linking the
first-passage dynamics to an explicit discrete optimization problem. Because
these decision quantities are computed by deterministic linear algebra,
simulation is required only for implementation checks and robustness studies,
not for evaluating the fitted finite-state model itself.

The numerical experiments illustrate the consequences of drift, boundary
size, cost asymmetry, calibration error, nonstationarity, and serial
dependence. Within the controlled synthetic settings, the exact computations
agree with Monte Carlo diagnostics, the intervention cost exhibits the
predicted single-crossing behaviour, and model-based schedules can improve the
specified cost rate relative to reactive or fixed-interval policies. These
findings should be interpreted as properties of the assumed discrete
dynamical system. They do not establish diagnostic accuracy, clinical benefit,
or superiority on real patient data.

The principal contribution is therefore methodological: the paper supplies an
auditable route from a two-sided discrete state process to first-passage
survival, RUL, and intervention quantities while retaining the lattice and
parity effects that continuous approximations may conceal. The framework also
identifies a broader research programme involving asymmetric and moving
absorbing sets, regime-modulated and nonautonomous chains, spectral and
quasi-stationary behaviour, multidimensional state spaces, and uncertainty-aware
calibration. These extensions would deepen the connection between finite-state
stochastic dynamics and condition-based decision-making in natural, social,
and engineered monitoring systems.

\section*{Acknowledgments}
This research received no external funding.

\section*{Declaration of interest statement}
There are no competing interests to declare.

\section*{Declaration of generative AI use}
During the preparation of this article, the author used Gemini 3 to refine the language. The content has been reviewed and edited by the author to ensure accuracy.

\section*{Data availability statement}
Data sharing is not applicable as this study does not generate or analyze data.

\section*{Appendix}
\appendix
\section{Algebraic Verification of the LR Inequality}\label{app:lr}

We verify~\eqref{eq:lr} for both parity skeletons. Define $a_i = p^{2i}+q^{2i}$ for $i\ge 0$.

\textbf{Even-$k$ skeleton.} From~\eqref{eq:even_prob}, for $i\ge 1$:
\[
p_{i,i+1} = \frac{a_{i+1}}{a_i},\quad p_{i,i}=2pq,\quad p_{i,i-1}=\frac{p^{2i}q^2+q^{2i}p^2}{a_i}.
\]
At state $i{+}1$, analogously $p_{i+1,i+1}=2pq$, $p_{i+1,i}=(p^{2(i+1)}q^2+q^{2(i+1)}p^2)/a_{i+1}$. Computing the two ratios:
\begin{align*}
\frac{p_{i+1,i+1}}{p_{i,i+1}} &= \frac{2pq\,a_i}{a_{i+1}},\\
\frac{p_{i+1,i}}{p_{i,i}} &= \frac{p^{2(i+1)}q^2+q^{2(i+1)}p^2}{2pq\,a_{i+1}}= \frac{p^2 q^2(p^{2i}+q^{2i})}{2pq\,a_{i+1}} = \frac{pq\,a_i}{2\,a_{i+1}}.
\end{align*}
where the second equality uses $p^{2(i+1)}q^2 + q^{2(i+1)}p^2 = p^2q^2(p^{2i}+q^{2i})$. Hence
$\displaystyle \frac{\text{LHS}}{\text{RHS}} = \frac{2pq\,a_i/a_{i+1}}{pq\,a_i/(2 a_{i+1})} = 4$.
The boundary case $i=0$ uses $p_{0,0}=2pq$, $p_{0,1}=a_1=p^2+q^2$, $p_{1,0}=(p^2q^2+q^2p^2)/a_1=2p^2q^2/a_1$, $p_{1,1}=2pq$, giving
$\displaystyle \frac{p_{1,1}}{p_{0,1}}\bigg/\frac{p_{1,0}}{p_{0,0}} = \frac{2pq/a_1}{2p^2q^2/(a_1\cdot 2pq)} = \frac{2pq/a_1}{pq/a_1} = 2$,
which still satisfies the inequality $\text{LHS}\ge\text{RHS}$ strictly.

\textbf{Odd-$k$ skeleton.} The odd-$k$ transitions for $i\ge 2$ involve $a^{\mathrm{odd}}_i = p^{2i-1}+q^{2i-1}$:
\[
p_{i,i+1} = \frac{a^{\mathrm{odd}}_{i+1}}{a^{\mathrm{odd}}_i},\quad p_{i,i}=2pq,\quad p_{i,i-1}=\frac{p^{2i-1}q^2+q^{2i-1}p^2}{a^{\mathrm{odd}}_i}.
\]
By the same factorisation $p^{2i+1}q^2+q^{2i+1}p^2 = p^2q^2 a^{\mathrm{odd}}_i$, the ratio is again exactly~4 for all $i\ge 2$. The boundary $i=1$ uses $p_{1,1}=3pq$, $p_{1,2}=(p^3+q^3)/(p+q)$, giving LHS/RHS $=6$, again strictly satisfying~\eqref{eq:lr}.

In every case the inequality~\eqref{eq:lr} holds strictly. By Lemma~\ref{prop:kijima}, the parity skeleton is IFR, hence $T_k$ is IFR (Proposition~\ref{prop:IFR}).

\bibliography{2reference}

\end{document}